\documentclass{amsart}

\usepackage{graphicx}
\usepackage{subfigure}
\usepackage{color}
\usepackage{pdflscape,array}

\newtheorem{theorem}{Theorem}[section]
\newtheorem{lemma}[theorem]{Lemma}

\theoremstyle{definition}

\newtheorem{example}[theorem]{Example}

\newtheorem{proposition}[theorem]{Proposition}
\newtheorem{problem}[theorem]{Problem}
\newtheorem{corollary}[theorem]{Corollary}

\theoremstyle{remark}

\numberwithin{equation}{section}

\begin{document}

\title[Connectivity and Irreducibility of FUNTF spaces]{Connectivity and Irreducibility of Algebraic Varieties of Finite Unit Norm Tight Frames}

\author{Jameson Cahill}
\address{Department of Mathematical Sciences, New Mexico State University, Las Cruces, New Mexico, 88003}
\email{jamesonc@nmsu.edu}
\author{Dustin G.\ Mixon}
\address{Department of Mathematics and Statistics, Air Force Institute of Technology, Wright-Patterson AFB, Ohio, 45433}
\email{dustin.mixon@afit.edu}

\author{Nate Strawn}
\address{Department of Mathematics and Statistics, Georgetown University, Washington, District of Columbia, 20007}
\email{nate.strawn@georgetown.edu}


\begin{abstract}
In this paper, we settle a long-standing problem on the connectivity of spaces of finite unit norm tight frames (FUNTFs), essentially affirming a conjecture first appearing in \cite{DS03}. Our central technique involves continuous liftings of paths from the polytope of eigensteps (see \cite{CFMPS12}) to spaces of FUNTFs. After demonstrating this connectivity result, we refine our analysis to show that the set of nonsingular points on these spaces is also connected, and we use this result to show that spaces of FUNTFs are irreducible in the algebro-geometric sense, and also that generic FUNTFs are full spark.
\end{abstract}

\maketitle

\section{Introduction}

\subsection{Background}
Frame theory began with the definition of frames by Duffin and Schaeffer \cite{DS52}, and today, frames provide a rich source of redundant representations and transformations. A \emph{frame} is a collection of vectors $\{f_n\}_{n\in I}$ in a Hilbert space $\mathcal{H}$ for which there exists strictly positive constants $A$ and $B$ satisfying
\[
A\Vert x\Vert_\mathcal{H}^2\leq\sum_{n\in I}\vert\langle x, f_n\rangle_{\mathcal{H}}\vert^2\leq B\Vert x\Vert_\mathcal{H}^2
\]
for all $x\in\mathcal{H}$, where $\langle\cdot,\cdot\rangle_\mathcal{H}$ is the inner product on $\mathcal{H}$ which induces the norm $\|\cdot\|_\mathcal{H}$. We call the frame \emph{tight} if we can take $A=B$. If the index set $I$ is finite, then $\mathcal{H}\cong \mathbb{F}^d$, where $\mathbb{F}=\mathbb{R}$ or $\mathbb{C}$, and if $\Vert f_n\Vert_\mathcal{H}=1$, we say that the frame is \emph{unit norm}. If the frame is finite, unit norm, and tight, we call it a \emph{finite unit norm tight frame (FUNTF)}. To put it simply, a FUNTF is the collection of column vectors in a matrix whose row vectors are orthogonal with equal norm, and whose column vectors each have unit norm.

Much of the early work on frames focused on infinite-dimensional frames: Fourier frames \cite{DS52}, Gabor frames \cite{BW94}, and wavelet frames \cite{DBAZ03}. In recent years, finite frames have been studied more rigorously because of their applications (for example, in wireless telecommunications \cite{SH03} and sigma-delta quantization \cite{BPY06}). While applications for frames abound, some of the most basic questions concerning frames remain unresolved. 

\subsection{The Frame Homotopy Problem}
The sets of real and complex FUNTFs of $N$ vectors in $d$ dimensions are denoted 
\[
\mathcal{F}_{N,d}^\mathbb{R}\text{ and } \mathcal{F}_{N,d}^\mathbb{C}
\]
respectively. The \emph{Frame Homotopy Problem} asks for which $N$ and $d$ these spaces are path-connected. Speculation on all the possible pairs of $N$ and $d$ for which path-connecivity holds was first formally enunciated in Conjectures 7.6 and 7.7 of \cite{DS03}, but the problem itself was first posed by D.\ R.\ Larson in a NSF Research Experiences for Undergraduates Summer Program in 2002. Though there are a large number of degrees of freedom in the spaces of FUNTFs, it has been surprisingly difficult to analytically construct anything but the simplest paths through these spaces. Moreover, many of these spaces have singularities which complicate the analysis of their geometry. 

The first step forward for the homotopy problem was shown in \cite{DS03}. The identification of FUNTFs in $\mathbb{R}^2$ with closed planar chains having links of length one in Theorem 3.3 of \cite{BF03} made it possible to obtain the connectivity result of \cite{DS03} using connectivity results for these chains \cite{LW95}. Such connectivity results were studied earlier because of the relevance with the well-studied problem of robotic motion planning. However, the analogue of the characterization in $\mathbb{R}^3$ is not so simple. This is because the identification in $\mathbb{R}^2$ is essentially obtained by the identification of the circle $\mathbb{S}^1$ with the real projective space $\mathbb{RP}^1$ via the map
\[
(x_1,x_2)\longmapsto \begin{pmatrix}
                            x_1^2 & x_1 x_2\\
                            x_1 x_2 & x_2^2
                            \end{pmatrix}.
\]
In $\mathbb{R}^3$, the sphere $\mathbb{S}^2$ is not identifiable with $\mathbb{RP}^2$. Closed chains are still releveant in this higher dimensional situation, but the configuration space is a subset of products of $\mathbb{RP}^2$ embedded into the space of $3\times 3$ symmetric matrices. This makes the motion much more difficult to imagine and the problem is no longer relatable to a well-studied area where connectivity results would have been considered. 

The most recent contribution to the homotopy problem came in 2009 with the work of Giol, Kovalev, Larson, Nguyen, and Tenner \cite{GKLNT09}. This work demonstrates the connectivity of certain families of projection operators which correspond to FUNTFs of $2d$ vectors in $d$ complex dimensions. Aside from this work, a lack of techniques for constructing paths has basically made it impossible to move forward on the problem over the last decade.

\subsection{Main Results}

In this paper, we completely resolve the Frame Homotopy Problem. In particular, we establish the following theorems:

\begin{theorem}\label{cmplxconnect}
The space $\mathcal{F}_{N,d}^{\mathbb{C}}$ is path-connected for all $d$ and $N$ satisfying $d\geq 1$ and $N\geq d$. 
\end{theorem}

\begin{theorem}\label{realconnect}
The space $\mathcal{F}_{N,d}^\mathbb{R}$ is path-connected for all $d$ and $N$ satisfying $d\geq 2$ and $N\geq d+2$.
\end{theorem}

The proofs of these results are distinct since the unitary group $\mathcal{U}(d)$ is connected and the orthogonal group $\mathcal{O}(d)$ is not. This makes the proof of Theorem \ref{cmplxconnect} substantially simpler. However, it is clear that Theorem \ref{realconnect} implies Theorem \ref{cmplxconnect} if any complex FUNTF is path-connected to some real FUNTF. While our methods may be used to produce such a reduction, we establish the results independently.

Let $\mathcal{G}^{\mathbb{F}}_{N,d}$ denote the set of rank $d$ orthogonal projections on $\mathbb{F}^N$ with all diagonal entries equal to $d/N$. There is a quotient map from $\mathcal{F}^{\mathbb{F}}_{N,d}$ to $\mathcal{G}^{\mathbb{F}}_{N,d}$ which preserves connectivity (see \cite{DS03}), so we also solve the main problem presented in \cite{GKLNT09} as a corollary of Theorems \ref{cmplxconnect} and \ref{realconnect}:

\begin{corollary}
The space $\mathcal{G}^{\mathbb{C}}_{N,d}$ is path-connected for all $N$ and $d$ satisfying $N\geq d\geq 1$ and $\mathcal{G}^{\mathbb{R}}_{N,d}$ is path-connected for all $N$ and $d$ satisfying $d\geq 2$ and $N\geq d+2$.
\end{corollary}

For $N$ and $d$ relatively prime, it was shown in \cite{DS03} that $\mathcal{F}_{N,d}^\mathbb{F}$ has no singularities for $\mathbb{F}=\mathbb{R}$ or $\mathbb{F}=\mathbb{C}$, and hence Theorems \ref{cmplxconnect} and \ref{realconnect} immediately imply that such an $\mathcal{F}_{N,d}^\mathbb{F}$ is an irreducible real algebraic variety. This raises the interesting possibility that the singularities in a general space of FUNTFs might be due to intersections of numerous irreducible components in $\mathcal{F}_{N,d}^\mathbb{F}$. Indeed, for $\mathcal{F}_{4,2}^\mathbb{R}$, this is evidently the case given the extensive analysis of this space in \cite{DS03}. However, by refining our connectivity result, we show that this is simply not the case in general, and that $\mathcal{F}_{4,2}^\mathbb{R}$ is exceptional in this regard.

\begin{theorem}\label{irreducible}
$\mathcal{F}_{N,d}^\mathbb{C}$ is an irreducible real algebraic variety for all $N$ and $d$ satisfying $N\geq d\geq 1$. $\mathcal{F}_{N,d}^\mathbb{R}$ is an irreducible real algebraic variety for all $N$ and $d$ satisfying $N\geq d+2\geq 4$, except when $N=4$ and $d=2$.
\end{theorem}

This result tells us that the singularities of spaces of FUNTFs either result from the space folding in on itself or developing cusps. Characterizing the local geometry around these singularities remains an open problem.

\begin{problem}
Describe the local geometry around singular points of $\mathcal{F}_{N,d}^\mathbb{F}$. 
\end{problem}

Finally, we use these methods to show that generic FUNTFs are \emph{full spark} (see \cite{ACM12}), i.e., they have the property that every subcollection of $d$ frame elements is linearly independent:

\begin{theorem}\label{fullspark}
A generic frame in $\mathcal{F}^{\mathbb{C}}_{N,d}$ is full spark for every $d$ and every $N\geq d$. A generic frame in $\mathcal{F}^{\mathbb{R}}_{N,d}$ is full spark for every $d$ and every $N\geq d$.
\end{theorem}

\subsection{Organization} 

In Section~2, we fix our notation and provide the background on necessary frame theory concepts. In Section~3, we establish our key technical tool, a lifting lemma for paths in spaces of eigensteps. The proofs of Theorems~\ref{cmplxconnect} and~\ref{realconnect} are demonstrated in Section~4. In Section \ref{sec:RAG} we discuss the real algebraic geometry of the FUNTF varieties. In particular, we introduce many definitions and results from real algebraic geometry, we determine the dimension of each $\mathcal{F}_{N,d}^\mathbb{F}$ as a real algebraic variety, we exhibit a dense subset of nonsingular points on these varieties, and we show that connectivity of this dense subset implies that $\mathcal{F}_{N,d}^\mathbb{F}$ is an irreducible real algebraic variety. In Section \ref{sec:nod-connect} we show when these dense subsets in the $\mathcal{F}_{N,d}^\mathbb{F}$ form a path-connected set, thereby demonstrating when $\mathcal{F}_{N,d}^\mathbb{F}$ is an irreducible real algebraic variety. In Section~7, we conclude with the consequences of our path-connectivity results, including the proofs of Theorems~\ref{irreducible} and~\ref{fullspark}.

\section{Prelimaries and Notation}

In general, we work with vectors in $\mathbb{F}^d$, where $\mathbb{F}$ is either the set $\mathbb{R}$ of real numbers or the set $\mathbb{C}$ of complex numbers. We assume that the inner product on $\mathbb{F}^d$ is the standard (symmetric or Hermitian) inner product. The set $M_{N,d}^\mathbb{F}$ consists of all $d$ by $N$ matrices (thought of as lists of columns) with entries in $\mathbb{F}$. For $X\in M_{N,d}^\mathbb{C}$, we let $\Re(X)\in M_{N,d}^\mathbb{R}$ and $\Im(X)\in M_{N,d}^\mathbb{R}$ denote the matrices obtained by taking the real and imaginary parts the entries, respectively. 

We let $I_k$ denote the $k\times k$ identity matrix, ${\bf 1}_k$ denotes the vector in $\mathbb{R}^k$ with entries all equal to $1$, and ${\bf 0}$ denotes a block of zero entries whose dimensions should be inferred from context. For any $k\times k$ matrix $A$, we use $\text{diag}(A)$ to denote the vector in $\mathbb{F}^k$ with entries equal to the diagonal entries of $A$ in order, and for any vector $v\in\mathbb{F}^k$, we let $\text{diag}(v)$ denote the $k\times k$ matrix with diagonal entries coinciding with the entries of $v$ and off-diagonal entries equal to zero. For a given collection of vectors $\{v_i\}_{i\in I}\subset \mathbb{F}^k$, we let $\text{span}\{v_i\}_{i\in I}$ denote the linear span of the collection, and we use $\text{span}^\perp\{v_i\}_{i\in I}$ to denote the orthogonal complement of the linear span in $\mathbb{F}^k$. We shall sometimes use $\{e_n\}_{n=1}^d$ to denote the standard orthonormal basis of $\mathbb{F}^k$.

For $k\geq 1$, we use $\mathcal{U}(k)$ to denote the Lie group of $k$ by $k$ unitary matrices, $\mathcal{O}(k)$ denote the Lie group of $k$ by $k$ orthogonal matrices, and $\mathcal{SO}(k)$ denotes the special orthogonal matrices (the matrices whose columns consist of positively-oriented orthonormal bases for $\mathbb{R}^k$). If $W$ is a subspace of $\mathbb{R}^k$, we may use $\mathcal{SO}(W)$ to denote the restriction of $\mathcal{SO}(k)$ to $W$, and we note that there is a canonical embedding $\mathcal{SO}(W)\hookrightarrow\mathcal{SO}(k)$.

For any integers $m$ and $n$ with $m\leq n$ we let $[m,n]$ denote the integers $\{m,m+1,\ldots, n\}$. We also abide by the convention that $[m,n]=\emptyset$ if $n<m$. Further assuming that $m\geq 1$ we let $[m]$ denote the first $m$ nonzero positive integers and we let $[m]_0$ denote $[m]\cup\{0\}$. 

\subsection{Frame Theory}

We let
\[
\mathcal{F}_{N,d}^\mathbb{F}
\]
denote the space of finite unit norm tight frames (FUNTFs) consisting of $N$ vectors in $\mathbb{F}^d$. Given an $F\in\mathcal{F}_{N,d}^\mathbb{F}$, we shall interchangeably identify $F$ with the indexed collection $\{f_n\}_{n\in[N]}$ and the $d\times N$ matrix of columns $\begin{pmatrix} f_1 & f_2 &\cdots& f_N\end{pmatrix}$. With a slight abuse of notation, if $F=\{f_n\}_{n=1}^{k_1}$ and $G=\{g_n\}_{n=1}^{k_2}$ are two ordered collections of vectors, we identify $F\cup G$ with the matrix $\begin{pmatrix} f_1 &\cdots& f_{k_1} & g_1 &\cdots & g_{k_1}\end{pmatrix}$. Exploiting this identification, the matrix $FF^\ast$ is called the \emph{frame operator} of $F$. 

Eigensteps \cite{CFMPS12} are the key tool used to derive our central technical lemmas. We shall let $\Lambda_{N,d}$ denote the space of FUNTF \emph{eigensteps}, or sequences $\lambda=\{\lambda_{n;i}\}_{i\in[d],n\in[N]_0}$ of nonincreasing sequences $\lambda_n$ satisfying
\begin{itemize}
\item[(i)] $\lambda_{0;i} = 0$ for all $i\in[d]$
\item[(ii)] $\lambda_{N;i}=N/d$ for all $i\in[d]$
\item[(iii)] $\lambda_n \sqsubseteq \lambda_{n+1}$ for all $n\in[N-1]$
\item[(iv)] $1+\sum_{i}\lambda_{n;i}=\sum_i\lambda_{n+1;i}$ for all $n\in[N-1]$
\end{itemize}
where $a\sqsubseteq b$ means that the interlacing inequalities 
\[
a_d\leq b_d\text{ and } b_{i+1}\leq a_i\leq b_i \text{ for }i\in[d-1]
\]
all hold. 

Note that $\Lambda_{N,d}$ is a convex polytope, and hence it is path-connected. We let $\operatorname{int}(\Lambda_{N,d})$ denote the relative interior of this polytope, which is characterized by Proposition \ref{strict-interlace}. If $\lambda$ satisfied the conditions of Proposition \ref{strict-interlace}, we say that it satisfies \emph{strict interlacing}. We also use  $\partial \Lambda_{N,d}=\Lambda_{N,d}\setminus\operatorname{int}(\Lambda_{N,d})$ to denote the relative boundary of $\Lambda_{N,d}$.

\begin{proposition}\label{strict-interlace}
If $\lambda\in\Lambda_{N,d}$, then $\lambda \in\operatorname{int}(\Lambda_{N,d})$ if and only if
\begin{align}
\lambda_{n;i}<\lambda_{n+1;i} \text{ for } n\in [i-1,N-d+i-1]\label{strict-upper}
\end{align}
holds for all $i\in[d]$ and
\begin{align}
\lambda_{n+1;i+1}<\lambda_{n;i} \text{ for } n\in [i,N-d+i-1]\label{strict-lower}
\end{align}
holds for all $i\in[d-1]$.
\end{proposition}
\begin{proof}
In Theorem 3.2 of \cite{HP15}, a point $\widehat{\lambda}\in \Lambda_{N,d}$ is constructed to satisfy these strict inequalities, and these inequalities constitute the maximal list of inequalities which may hold without equality.\end{proof}

For a given frame $F\in\mathcal{F}_{N,d}^\mathbb{F}$, we shall let $\Lambda(F)$ denote the eigensteps associated with $F$. That is, $\{{\bf 0}\}\cup\{\{\lambda_{n;i}(F)\}_{i\in[d]}\}_{n\in[N]}$ is the set of eigenvalues (counting multiplicity and nonincreasing in the index $i$) of the operator 
\[
\sum_{n=1}^k f_n f_n^\ast,
\]
which is referred to as the $k$th partial frame operator of $F$. Note that $\Lambda:\mathcal{F}^{\mathbb{F}}_{N,d}\rightarrow\Lambda_{N,d}$ is a well defined mapping, but it is not injective so there can be many frames that have the same eigensteps. We shall sometime refer to $\Lambda$ as the eigensteps map.

A frame $F=\{f_n\}_{n\in[N]}$ is said to be \emph{orthodecomposable} (OD, pronounced ``odd'') if there is a nontrivial disjoint partition of $[N]$ into $S$ and $T$ such that 
\[
\text{span}\{f_n\}_{n\in S} = \text{span}^\perp\{f_n\}_{n\in T}.
\]
As we shall soon see, the set of OD frames contain the singular points of the FUNTF varieties, and the nonorthodecomposable (NOD, pronounced ``nod") frames are contained in the set of nonsingular points of the FUNTF varieties.

The \emph{correlation network} was introduced in \cite{S2007} to provide a simple characterization of OD frames, and certain arguments dramatically simplify when they are expressed in terms of the correlation networks. The correlation network (now sometimes known as the frame graph) of a frame $\gamma(F)=(V,E)$ is the undirected graph with vertex set $V=[N]$ such that $(i,j)\in E$ if and only if $\langle f_i,f_j\rangle\not = 0$.

\begin{proposition}[Lemma~3.2.5 in~\cite{S2007}]\label{corrnet}
A frame $F$ is NOD if and only if $\gamma(F)$ is connected.
\end{proposition}

Given a frame $F\in\mathcal{F}^{\mathbb{F}}_{N,d}$, we define the \textit{spark} of $F$ (written $\mathrm{spark}(F)$) to be the size of the smallest linearly dependent subset of $F$.  Note that $\mathrm{spark}(F)\leq d+1$; if $\mathrm{spark}(F)=d+1$ then we say that $F$ is \textit{full spark}.  Observe that a frame is full spark if and only if the following statement holds: If $k<d$ and $W\subseteq\mathbb{F}^d$ is a subspace of dimension $k$, then $|W\cap F|\leq k$. Therefore it follows that any orthodecomposable frame with $N>d$ is not full spark.

Our final essential ingredient is the Naimark complement of a Parseval frame \cite{HL00}. A frame $F$ for $\mathbb{F}^d$ is said to be a \emph{Parseval frame} if $FF^*=I_d$ (which is equivalent to $A=B=1$ for the optimal frame bounds $A$ and $B$ -- see the first chapter of \cite{CK13}). This is equivalent to $F^*F$ being an orthogonal projection. A frame $G$ for $\mathbb{F}^{N-d}$ is called a \emph{Naimark complement} of a Parseval frame $F$ if $F^*F+G^*G=I_N$ . Naimark complements preserve many important properties of frames.

\begin{proposition}\label{naimark}
Suppose $F$ and $G$ are Naimark complementary Parseval frames. Then
\begin{itemize}
\item[(i)] $F$ is equal norm if and only if $G$ is equal norm,
\item[(ii)] $F$ is OD if and only if $G$ is OD, and
\item[(iii)] $F$ is full spark if and only if $G$ is full spark.
\end{itemize}
\end{proposition}

Here, (i) follows from considering the diagonal entries of $F^*F$ and $G^*G$, (ii) follows from Proposition~\ref{corrnet} since $\gamma(F)=\gamma(G)$, while (iii) is far less obvious (see Theorem~4(iii) in~\cite{ACM12} for the essential ingredients of the proof). Noting that a FUNTF is a scalar multiple of a Parseval frame, a series of rescalings allows us to extend the notion of Naimark complements to FUNTFs. The following is an important consequence of the Naimark complement:

\begin{proposition}[Corollary 7.3 in \cite{DS03}]\label{equivconnect}
When $N> d$, $\mathcal{F}_{N,d}^\mathbb{F}$ is connected if and only if $\mathcal{F}_{N,N-d}^\mathbb{F}$ is connected.
\end{proposition}

With Proposition~\ref{naimark}(ii), it is easy to imitate the argument for the above result to get a similar result for NOD frames.

\begin{proposition}\label{nodequivconnect}
When $N> d$, the set of NOD frames in $\mathcal{F}_{N,d}^\mathbb{F}$ is connected if and only if the set of NOD frames in $\mathcal{F}_{N,N-d}^\mathbb{F}$ is connected.
\end{proposition}

\section{Lifting paths in $\Lambda_{N,d}$ to $\mathcal{F}_{N,d}^\mathbb{F}$}

This section provides the technical lemmas involving eigensteps that we will exploit throughout the remainder of the paper. The main idea behind these lemmas is that paths in the eigensteps polytope $\Lambda_{N,d}$ can be lifted to paths of frames in $\mathcal{F}_{N,d}^\mathbb{F}$ in such a way that the eigensteps of each frame in the frame path are given by the corresponding point in the eigensteps path.

In order to construct these paths, we first show that Theorem 7 of \cite{CFMPS12} greatly simplifies for the frames $\mathcal{F}_{N,d}^\mathbb{F}$ when considering eigensteps on the relative interior of $\Lambda_{N,d}$. We now state the simplified result and present an example of the simplified structure in Figure \ref{fig:esteps-example}.

\begin{theorem}[c.f. Theorem 7 of \cite{CFMPS12}]\label{interior-eigensteps}
Suppose $\lambda\in\operatorname{int}(\Lambda_{N,d})$, set $\kappa(n)=n-N+d+1$ for $n\geq N-d$, and define the $d$ by $N-1$ matrices $v(\lambda)$ and $w(\lambda)$, and the sequence of $d$ by $d$ matrices $W_n(\lambda)$ for $n\in[N-1]$ by fixing the entries
\small
\[
v_{n;i}(\lambda) =
\left\{\begin{array}{cl}
\left[-\frac{\prod_{j\in[n+1]}(\lambda_{n;i}-\lambda_{n+1; j})}{\prod_{j\in[n+1]\setminus\{i\}} (\lambda_{n;i}-\lambda_{n;j})}\right]^{1/2} & \text{ for }n\in[d-1], i\in[n+1]\\
\left[-\frac{\prod_{j\in[d]}(\lambda_{n;i}-\lambda_{n+1; j})}{\prod_{j\in[d]\setminus\{i\}} (\lambda_{n;i}-\lambda_{n;j})}\right]^{1/2} & \text{ for }n\in[d,N-d], i\in[d]\\
\left[-\frac{\prod_{j\in[\kappa(n),d]}(\lambda_{n;i}-\lambda_{n+1; j})}{\prod_{j\in[\kappa(n),d]\setminus\{i\}} (\lambda_{n;i}-\lambda_{n;j})}\right]^{1/2} & \text{ for }n\in[N-d+1,N-1], i\in[\kappa(n),d]\\
0 & \text{ otherwise }
\end{array}\right.
\]
\[
w_{n;i}(\lambda)=\left\{\begin{array}{cl}
\left[-\frac{\prod_{j\in[n+1]}(\lambda_{n+1; i}-\lambda_{n;j})}{\prod_{j\in[n+1]\setminus\{i\}} (\lambda_{n+1;i}-\lambda_{n+1;j})}\right]^{1/2} & \text{ for }n\in[d-1], i\in[n+1]\\
\left[-\frac{\prod_{j\in[d]}(\lambda_{n+1;i}-\lambda_{n; j})}{\prod_{j\in[d]\setminus\{i\}} (\lambda_{n+1;i}-\lambda_{n+1;j})}\right]^{1/2} & \text{ for }n\in[d,N-d], i\in[d]\\
\left[-\frac{\prod_{j\in[\kappa(n),d]}(\lambda_{n+1;i}-\lambda_{n; j})}{\prod_{j\in[\kappa(n),d]\setminus\{i\}} (\lambda_{n+1;i}-\lambda_{n+1;j})}\right]^{1/2} & \text{ for }n\in[N-d+1,N-1], i\in[\kappa(n),d]\\
0 & \text{ otherwise }
\end{array}\right.
\]\normalsize
and
\small
\[
(W_n)_{j;i}(\lambda) = \left\{\begin{array}{cl}
 \frac{v_{n;i}(\lambda)w_{n;j}(\lambda)}{\lambda_{n+1;j}-\lambda_{n;i}} & \text{ when } n\in [i-1,N-d+i-1]\cap[j-1,N-d+j-1]\\
 1 &\text{ when } i=j\text{ and }n\not\in [i-1,N-d+i-1]^2\\
 0 &\text{ otherwise }
 \end{array}\right.
\]\normalsize
If $U_1$ is a $d$ by $d$ unitary matrix, and $V_n$ is a sequence of diagonal $d$ by $d$ unitary matrices for $n\in[N-1]$, then the sequence $\{f_n\}_{n=1}^N\subset\mathbb{F}^d$ defined by
\begin{enumerate}
\item $f_1=u_1$ (the first column of $U_1$)
\item $f_{n+1} = U_nV_n v_n$ and $U_{n+1}=U_n V_n W_n$ for $n\in[N-1]$
\end{enumerate}
is such that $F=\begin{pmatrix} f_1 & f_2 &\cdots & f_N\end{pmatrix}\in \mathcal{F}_{N,d}^\mathbb{F}$.

\end{theorem}

\begin{landscape}
\begin{figure}
\begin{center}
{\setlength{\tabcolsep}{12pt}
\setlength{\extrarowheight}{30pt}
\begin{tabular}{c|ccc}
$n$ & $v_n(\lambda)$ & $w_n(\lambda)$ & $W_n(\lambda)$ \\[16pt]\hline
1 & ${\setlength{\extrarowheight}{0pt} \begin{pmatrix}
\sqrt{ -\frac{(\lambda_{11}-\lambda_{12})(\lambda_{11}-\lambda_{22})}{\lambda_{11}-\lambda_{21}}}\\
\sqrt{ -\frac{(\lambda_{21}-\lambda_{12})(\lambda_{21}-\lambda_{22})}{\lambda_{21}-\lambda_{11}}}\\
0
\end{pmatrix}}$ & ${\setlength{\extrarowheight}{0pt}\begin{pmatrix}
\sqrt{\frac{(\lambda_{12}-\lambda_{11})(\lambda_{12}-\lambda_{21})}{\lambda_{12}-\lambda_{22}}}\\
\sqrt{\frac{(\lambda_{22}-\lambda_{11})(\lambda_{22}-\lambda_{21})}{\lambda_{22}-\lambda_{12}}}\\
0
\end{pmatrix}}$ & ${\setlength{\extrarowheight}{0pt}\begin{pmatrix}
\frac{v_{11}(\lambda)w_{11}(\lambda)}{\lambda_{12}-\lambda_{11}} & \frac{v_{11}(\lambda)w_{21}(\lambda)}{\lambda_{22}-\lambda_{11}} &0\\[6pt]
\frac{v_{21}(\lambda)w_{11}(\lambda)}{\lambda_{12}-\lambda_{21}} & \frac{v_{21}(\lambda)w_{21}(\lambda)}{\lambda_{22}-\lambda_{21}} &0\\[6pt]
0& 0 &1
\end{pmatrix}}$\\ 
2 & ${\setlength{\extrarowheight}{0pt}\begin{pmatrix}
\sqrt{ -\frac{(\lambda_{12}-\lambda_{13})(\lambda_{12}-\lambda_{23})(\lambda_{12}-\lambda_{33})}{(\lambda_{12}-\lambda_{22})(\lambda_{12}-\lambda_{32})}}\\
\sqrt{ -\frac{(\lambda_{22}-\lambda_{13})(\lambda_{22}-\lambda_{23})(\lambda_{22}-\lambda_{33})}{(\lambda_{22}-\lambda_{12})(\lambda_{22}-\lambda_{32})}}\\
\sqrt{ -\frac{(\lambda_{32}-\lambda_{13})(\lambda_{32}-\lambda_{23})(\lambda_{32}-\lambda_{33})}{(\lambda_{32}-\lambda_{12})(\lambda_{32}-\lambda_{22})}}
\end{pmatrix}}$ & ${\setlength{\extrarowheight}{0pt}\begin{pmatrix}
\sqrt{\frac{(\lambda_{13}-\lambda_{12})(\lambda_{13}-\lambda_{22})(\lambda_{13}-\lambda_{32})}{(\lambda_{13}-\lambda_{23})(\lambda_{13}-\lambda_{33})}}\\
\sqrt{\frac{(\lambda_{23}-\lambda_{12})(\lambda_{23}-\lambda_{22})(\lambda_{23}-\lambda_{32})}{(\lambda_{23}-\lambda_{13})(\lambda_{23}-\lambda_{33})}}\\
\sqrt{\frac{(\lambda_{33}-\lambda_{12})(\lambda_{33}-\lambda_{22})(\lambda_{33}-\lambda_{32})}{(\lambda_{33}-\lambda_{13})(\lambda_{33}-\lambda_{23})}}
\end{pmatrix}}$ & ${\setlength{\extrarowheight}{0pt}\begin{pmatrix}
\frac{v_{12}(\lambda)w_{12}(\lambda)}{\lambda_{13}-\lambda_{12}} & \frac{v_{12}(\lambda)w_{22}(\lambda)}{\lambda_{23}-\lambda_{12}} &\frac{v_{12}(\lambda)w_{32}(\lambda)}{\lambda_{33}-\lambda_{12}}\\[6pt]
\frac{v_{22}(\lambda)w_{12}(\lambda)}{\lambda_{13}-\lambda_{22}} & \frac{v_{22}(\lambda)w_{22}(\lambda)}{\lambda_{23}-\lambda_{22}} &\frac{v_{22}(\lambda)w_{32}(\lambda)}{\lambda_{33}-\lambda_{22}}\\[6pt]
\frac{v_{32}(\lambda)w_{12}(\lambda)}{\lambda_{13}-\lambda_{32}} & \frac{v_{32}(\lambda)w_{22}(\lambda)}{\lambda_{23}-\lambda_{32}} &\frac{v_{32}(\lambda)w_{32}(\lambda)}{\lambda_{33}-\lambda_{32}}
\end{pmatrix}}$ \\
3 & ${\setlength{\extrarowheight}{0pt}\begin{pmatrix}
\sqrt{ -\frac{(\lambda_{13}-\lambda_{14})(\lambda_{13}-\lambda_{24})(\lambda_{13}-\lambda_{34})}{(\lambda_{13}-\lambda_{23})(\lambda_{13}-\lambda_{33})}}\\
\sqrt{ -\frac{(\lambda_{23}-\lambda_{14})(\lambda_{23}-\lambda_{24})(\lambda_{23}-\lambda_{34})}{(\lambda_{23}-\lambda_{13})(\lambda_{23}-\lambda_{33})}}\\
\sqrt{ -\frac{(\lambda_{33}-\lambda_{14})(\lambda_{33}-\lambda_{24})(\lambda_{33}-\lambda_{34})}{(\lambda_{33}-\lambda_{13})(\lambda_{33}-\lambda_{23})}}\\
\end{pmatrix}}$ & ${\setlength{\extrarowheight}{0pt}\begin{pmatrix}
\sqrt{\frac{(\lambda_{14}-\lambda_{13})(\lambda_{14}-\lambda_{23})(\lambda_{14}-\lambda_{33})}{(\lambda_{14}-\lambda_{24})(\lambda_{14}-\lambda_{34})}}\\
\sqrt{\frac{(\lambda_{24}-\lambda_{13})(\lambda_{24}-\lambda_{23})(\lambda_{24}-\lambda_{33})}{(\lambda_{24}-\lambda_{14})(\lambda_{24}-\lambda_{34})}}\\
\sqrt{\frac{(\lambda_{34}-\lambda_{13})(\lambda_{34}-\lambda_{23})(\lambda_{34}-\lambda_{33})}{(\lambda_{34}-\lambda_{14})(\lambda_{34}-\lambda_{24})}}\\
\end{pmatrix}}$ & ${\setlength{\extrarowheight}{0pt}\begin{pmatrix}
\frac{v_{13}(\lambda)w_{13}(\lambda)}{\lambda_{14}-\lambda_{13}} & \frac{v_{13}(\lambda)w_{23}(\lambda)}{\lambda_{24}-\lambda_{13}} &\frac{v_{13}(\lambda)w_{33}(\lambda)}{\lambda_{34}-\lambda_{13}}\\[6pt]
\frac{v_{23}(\lambda)w_{13}(\lambda)}{\lambda_{14}-\lambda_{23}} & \frac{v_{23}(\lambda)w_{23}(\lambda)}{\lambda_{24}-\lambda_{23}} &\frac{v_{23}(\lambda)w_{33}(\lambda)}{\lambda_{34}-\lambda_{23}}\\[6pt]
\frac{v_{33}(\lambda)w_{13}(\lambda)}{\lambda_{14}-\lambda_{33}} & \frac{v_{33}(\lambda)w_{23}(\lambda)}{\lambda_{24}-\lambda_{33}} &\frac{v_{33}(\lambda)w_{33}(\lambda)}{\lambda_{34}-\lambda_{33}}
\end{pmatrix}}$\\
4 & ${\setlength{\extrarowheight}{0pt}\begin{pmatrix}
0\\
\sqrt{ -\frac{(\lambda_{24}-\lambda_{25})(\lambda_{24}-\lambda_{35})}{\lambda_{24}-\lambda_{34}}}\\
\sqrt{ -\frac{(\lambda_{34}-\lambda_{25})(\lambda_{34}-\lambda_{35})}{\lambda_{34}-\lambda_{24}}}
\end{pmatrix}}$ & ${\setlength{\extrarowheight}{0pt}\begin{pmatrix}
0\\
\sqrt{\frac{(\lambda_{25}-\lambda_{24})(\lambda_{25}-\lambda_{34})}{\lambda_{25}-\lambda_{35}}}\\
\sqrt{\frac{(\lambda_{35}-\lambda_{24})(\lambda_{35}-\lambda_{34})}{\lambda_{35}-\lambda_{25}}}
\end{pmatrix}}$ & ${\setlength{\extrarowheight}{0pt}\begin{pmatrix}
1 & 0 &0\\[6pt]
0&\frac{v_{24}(\lambda)w_{24}(\lambda)}{\lambda_{25}-\lambda_{24}} & \frac{v_{24}(\lambda)w_{34}(\lambda)}{\lambda_{35}-\lambda_{24}} \\[6pt]
0&\frac{v_{34}(\lambda)w_{24}(\lambda)}{\lambda_{25}-\lambda_{34}} & \frac{v_{34}(\lambda)w_{34}(\lambda)}{\lambda_{35}-\lambda_{34}}
\end{pmatrix}}$\\
5& ${\setlength{\extrarowheight}{0pt}\begin{pmatrix}
0\\
0\\
1
\end{pmatrix}}$ & ${\setlength{\extrarowheight}{0pt}\begin{pmatrix}
0\\
0\\
1
\end{pmatrix}}$ & ${\setlength{\extrarowheight}{0pt}\begin{pmatrix}
1 & 0 & 0\\
0 & 1 & 0\\
0 & 0 & 1
\end{pmatrix}}$
\end{tabular} }
\end{center}
\caption{The $v_n$, $w_n$, and $W_n$ of Theorem \ref{interior-eigensteps} when $N=6$ and $d=3$.}
\label{fig:esteps-example}
\end{figure}
\end{landscape}

\begin{proof}
Using the notation from Theorem 7 of \cite{CFMPS12}, the index sets $\mathcal{I}_n\subset[d]$ and $\mathcal{J}_n\subset[d]$ are defined as follows:
\begin{itemize}
\item $\mathcal{I}_n\subseteq[d]$ consists of the indices $i$ satisfying the inclusion critera:
\begin{enumerate}
\item $\lambda_{n;i} < \lambda_{n;j}$ for all $j<i$;
\item and the multiplicity of $\lambda_{n;i}$ as a value in $\{\lambda_{n;j}\}_{j=1}^d$ exceeds its multiplicity as a value in $\{\lambda_{n+1;j}\}_{j=1}^d$.
\end{enumerate}
\item $\mathcal{J}_n\subseteq[d]$ consists of the indices $i$ satisfying the inclusion critera:
\begin{enumerate}
\item $\lambda_{n+1;i} < \lambda_{n+1;j}$ for all $j<i$;
\item and  the multiplicity of $\lambda_{n+1;i}$ as a value in $\{\lambda_{n+1;j}\}_{j=1}^d$ exceeds its multiplicity as a value in $\{\lambda_{n;j}\}_{j=1}^d$.
\end{enumerate}
\end{itemize}
The following remarks should help demystify the role of $\mathcal{I}_n$ and $\mathcal{J}_n$, and how they are expected to simplify on $\operatorname{int}(\Lambda_{N,d})$.
\begin{enumerate}
\item The sets $\mathcal{I}_n$ are attempting to identify the eigenspaces of the partial frame operators $\sum_{i=1}^n f_i f_i^\ast$ that are ``shrinking" as $n$ increases. This ``shrinking" is effectively tracked by considering if the multiplicity of a particular eigenvalue is smaller at the next step, where eigenvalues not in the next spectrum are assigned multiplicity 0.
\item The sets $\mathcal{J}_n$ are attempting to identify the eigenspaces of the partial frame operators that are ``growing" as $n$ increases. This ``growth" is effectively tracked by considering if the multiplicity of a particular eigenvalue is larger than it was at the last step, where eigenvectors not in the previous spectrum are assigned multiplicity 0.
\item On $\operatorname{int}(\Lambda_{N,d})$, the strict interlacing inequalities essentially imply that successive spectra of the partial frame operators will be disjoint if we exclude the eigenvalues 0 and $\frac{N}{d}$. Thus, $\mathcal{I}_n=\mathcal{J}_n=[d]$ except for when $n<d$ (when the eigenspace corresponding to the eigenvalue $0$ is shrinking as $n$ increases) or $n>N-d$ (when the eigenspace corresponding to the eigenvalue $\frac{N}{d}$ is growing as $n$ increases).
\end{enumerate}

We will now show that 
\[
\mathcal{I}_n=\left\{\begin{array}{cl}
\{1,\ldots,n+1\} & \text{ if } n<d\\
\{1,\ldots,d\} & \text{ if }d\leq n \leq N-d\\
\{\kappa(n),\ldots, d\} & \text{ if } N-d<n\leq N-1
\end{array}\right.
\]
and 
\[
\mathcal{J}_n=\left\{\begin{array}{cl}
\mathcal{I}_n & \text{ if } n \leq N-d\\
\{1,n-N+d+2,\ldots, d\} & \text{ if } N-d<n<N-1\\
\{1\} & \text{ if }n=N-1
\end{array}\right.
\]
if $\lambda\in \operatorname{int}(\Lambda_{N,d})$. We shall do this by demonstrating that $i\in[d]$ belongs to $\mathcal{I}_n$ if and only if $n\in[i-1,N-d+i-1]$, and $i\in[d]$ belongs to $\mathcal{J}_n$ if and only if $n\in[i-1,N-d+i-2]$ or $i=1$. To do so, we make extensive use of the strict interlacing inequalities of Proposition \ref{strict-interlace}.

We observe that $\lambda_{n;1}<\lambda_{n;j}$ and $\lambda_{n+1;1}<\lambda_{n+1;j}$ holds vacuously for all $j<1$. Now assuming that $i>1$, we have that 
\[
\lambda_{n;i}\leq \lambda_{n-1;i-1}<\lambda_{n;i-1}\text{ for all } n-1\in [i-2,N-d+i-2]\text{ with }n-1\geq 0
\]
and
\[
\lambda_{n+1;i}\leq \lambda_{n;i-1}<\lambda_{n+1;i-1}\text{ for all } n\in[i-2,N-d+i-2]\text{ with }n\geq 0.
\]
by (\ref{strict-upper}). Using the fact that the entries of $\lambda_n$ are ordered for each $n$, we have that the first inclusion criterion for $\mathcal{I}_n$ is satisfied for all index pairs $(n;i)$ satisfying $i=1$ or $i\in[2,d]$ and $n\in[i-1,N-d+i-1]$, and also that the first inclusion criterion for $\mathcal{J}_n$ is satisfied for all index pairs $(n;i)$ satisfying $i=1$ or $i\in[2,d]$ and $n\in[i-2,N-d+i-2]$ with $n\geq 0$. Also note that these inequalities and the fact that the entries of $\lambda_{n}$ are ordered for each $n$ imply that $\lambda_{n;d}$ has multiplicity $1$ in $\{\lambda_{n;j}\}_{j=1}^d$ for $n\geq d-1$.

We now focus on excluding indices from the $\mathcal{I}_n$. If $n<i-1$, then $\lambda_{n;i}=\lambda_{n;n-1}=0$, so the first inclusion criterion fails. If $n\in[N-d+i,N-1]$, then $\lambda_{n;i}=\frac{N}{d}$, and since $\frac{N}{d}$ has multiplicity $n-(N-d)$ in $\{\lambda_{n;j}\}_{j=1}^d$ but multiplicity $n+1-(N-d)$ in $\{\lambda_{n+1;j}\}_{j=1}^d$, the second inclusion criterion fails for these pairs of indices. Note this also excludes the cases where $i=1$ but $n\geq N-d+1$. 

The remaining step is to show that the pairs $(n;i)$ satisfying $i\in[d]$ and $n\in[i-1,N-d+i-1]$ satisfy the second inclusion criterion for $\mathcal{I}_n$. If $n=i-1$, we have that $\lambda_{n;i}=0$. Noting that $0$ has multiplicity $d-i$ in $\lambda_n$ and multiplicity $d-i-1$ in $\lambda_{n+1}$, we conclude that the second inclusion criterion is satisfied for these cases. Finally, for all $i\in[d-1]$ and $n\in[i,N-d+i-1]$, the inequalities
\[
\lambda_{n;i+1}\leq \lambda_{n+1;i+1}<\lambda_{n;i}
\]
hold by (\ref{strict-lower}). Coupling this with our note about the multiplicity of $\lambda_{n;d}$, we get that the multiplicity of $\lambda_{n;i}$ in $\lambda_n$ is $1$ for all $n\in[i,N-d+i-1]$. On the other hand, the entries of $\lambda_n$ are ordered, and so the bounds
\[
\lambda_{n+1;i+1} < \lambda_{n;i} < \lambda_{n+1;i}\text{ for all } n\in[i,N-d+i-1]
\]
when $i\in[d-1]$ and
\[
\lambda_{n;d} < \lambda_{n+1;d}\text{ for all }n\in[d,N-1]
\]
imply that the multiplicity of $\lambda_{n;i}$ in $\lambda_{n+1}$ is $0$ for all $n\in[i,N-d+i-1]$. Therefore, we conclude that the second inclusion criterion is satisfied for all of these points, and we have that $\mathcal{I}_n$ consists of all $i\in[d]$ satisfying $n\in[i-1,N-d+i-1]$.

We now perform a similar demonstration for $\mathcal{J}_n$, and we begin by excluding indices. If $i\in[2,d]$ and $n\in[N-d+i-1,N-1]$, then $n+1\in[N-d+i,N]$ and hence $\lambda_{n+1;i}=\frac{N}{d}$. But then $\lambda_{n+1;1}=\frac{N}{d}$, and the first inclusion criteria for $\mathcal{J}_n$ fails for these indices. On the other hand, if $n=i-2$ for $i\in[2,d]$, we know that $\lambda_{n+1; n+2}=0$. Since $0$ has multiplicity $d-n-1$ in $\lambda_{n+1}$ and multiplicity $d-n$ in $\lambda_{n;j}$, the second inclusion criterion for $\mathcal{J}_n$ fails for these indices. We now proceed to show that the indices $(n;i)$ satisfying $i=1$ or $i\in[2,d]$ and $n\in[i-1,N-d+i-2]$ satisfy the second inclusion criterion for $\mathcal{J}_n$.

First note that if $n\in[N-d,N-1]$, then $\lambda_{n+1;1}=\frac{N}{d}$ has multiplicity $n+1-(N-d)$ in $\lambda_{n+1}$ and multiplicity $n-(N-d)$ in $\lambda_{n}$, and therefore the second inclusion criterion for $\mathcal{J}_n$ is satisfied by these indices. Now observing that
\[
\lambda_{n+1;i+1}\leq \lambda_{n;i} < \lambda_{n+1;i}\text{ for all }n\in[i-1,N-d+i-2]
\]
when $i\in[d-1]$, the ordering of the entries of $\lambda_{n+1}$ gives us that $\lambda_{n+1;1}$ has multiplicity $1$ in $\lambda_{n+1}$, and further observing that 
\[
\lambda_{n+1;i}\leq \lambda_{n;i-1}<\lambda_{n+1;i-1}\text{ for all }n\in[i-1,N-d+i-2]
\]
when $i\in[2,d]$, we have that $\lambda_{n+1;i}$ has multiplicity $1$ in $\lambda_{n+1}$ for all $n\in[i-1,N-d+i-2]$ when $i\in[d]$. The inequalities
\[
\lambda_{n;1}<\lambda_{n+1;1} \text{ for all } n\in[1,N-d-1]
\]
and the ordering of the entries of $\lambda_{n}$ together imply that $\lambda_{n+1;1}$ has multiplicity 0 in $\lambda_{n}$ for all $n\in[0,N-2]$, and so the second inclusion criterion is satisfied. For $i\in[2,d]$, we have that
\[
\lambda_{n;i}<\lambda_{n+1;i}<\lambda_{n;i-1} \text{ for } n\in[i-1,N-d+i-2]
\]
and the ordering of the entries of $\lambda_{n}$ together imply that $\lambda_{n+1;i}$ has multiplicity 0 in $\lambda_{n}$ for these indices. Thus, the second inclusion criterion for $\mathcal{J}_n$ is satisfied for all index pairs $(n;i)$ satisfying $i\in[d]$ and $n\in[i-1,N-d+i-2]$. Aggregating all of our results together, we have that $\mathcal{J}_n$ consists of the index $1$ along with any $i\in[2,d]$ satisfying $n\in[i-1,N-d+i-2]$.

Now that we have identified $\mathcal{I}_n$ and $\mathcal{J}_n$, we recall the definitions of the permutations $\pi_{\mathcal{I}_n}$ and $\pi_{\mathcal{J}_n}$ from Theorem 7 of \cite{CFMPS12}:
\begin{itemize}
\item First, recall that a permutation $\pi$ on a set $[d]$ is said to be increasing on $A\subseteq[d]$ if $\pi(i)\leq\pi(j)$ for all $i,j\in A$ with $i\leq j$.
\item $\pi_{\mathcal{I}_n}$ is the unique permutation that is increasing on both of the sets $\mathcal{I}_n$ and $[d]\setminus\mathcal{I}_n$, and such that $\pi_{\mathcal{I}_n}(\mathcal{I}_n)=[\vert\mathcal{I}_n\vert]$.
\item Similarly, $\pi_{\mathcal{J}_n}$ is the unique permutation that is increasing on both of the sets $\mathcal{J}_n$ and $[d]\setminus\mathcal{J}_n$, and such that $\pi_{\mathcal{J}_n}(\mathcal{J}_n)=[\vert\mathcal{J}_n\vert]$.
\end{itemize}
Under our hypothesis, we note that the associated permutation matrices $\Pi_{\mathcal{I}_n}$ and $\Pi_{\mathcal{J}_n}$ defined in Theorem 7 of \cite{CFMPS12} satisfy $\Pi_{\mathcal{I}_n}=\Pi_{\mathcal{J}_n}=I_d$ when $n\leq N-d$,   
\[
\Pi_{\mathcal{I}_n} = \begin{pmatrix}
{\bf 0} & I_{N-n}\\
I_{n-N+d} & {\bf 0}
\end{pmatrix}
\]
when $N-d<n<N$,
\[ \Pi_{\mathcal{J}_n} =\begin{pmatrix}
1 & {\bf 0} & {\bf 0}\\
{\bf 0} & {\bf 0} & I_{N-n-1}\\
{\bf 0} & I_{n-N+d} & {\bf 0}
\end{pmatrix}
\]
when $N-d<n<N-1$, and $\Pi_{\mathcal{J}_{N-1}}=I_d$.

We now claim that
\[
v_n(\lambda) = \Pi_{\mathcal{I}_n}^T\begin{pmatrix}v_n\\ {\bf 0} \end{pmatrix},\: w_n(\lambda) = \Pi_{\mathcal{I}_n}^T\begin{pmatrix}w_n\\ {\bf 0}\end{pmatrix},
\]
and 
\[
W_n(\lambda) = \Pi_{\mathcal{I}_n}^T\begin{pmatrix} W_n & {\bf 0} \\
															{\bf 0} & I_k\end{pmatrix}\Pi_{\mathcal{I}_n}
\]
for all $n\in[N-1]$, where $v_n$, $w_n$, and $W_n$ defined in Theorem 7 of \cite{CFMPS12} are the counterparts the $v_n(\lambda)$, $w_n(\lambda)$, and $W_n(\lambda)$ that we have defined. 

For $n\leq d-1$, we observe that 
\begin{align*}
v_{n;i}(\lambda)&=\sqrt{-\frac{\prod_{j\in[n+1]} (\lambda_{n;i}-\lambda_{n+1;j})}{\prod_{j\in[n+1]\setminus\{i\}}(\lambda_{n;i}-\lambda_{n;j})}}=\sqrt{-\frac{\prod_{j\in\mathcal{J}_n} (\lambda_{n;i}-\lambda_{n+1;j})}{\prod_{j\in\mathcal{I}_n\setminus\{i\}}(\lambda_{n;i}-\lambda_{n;j})}}=v_{n;i}
\end{align*}
if $i\in[n+1]$, and $v_{n;i}(\lambda)=0$ for $i\in[d]$. For $n\in[d,N-d]$, the equivalence is obvious. If $n\in[N-d+1,N-1]$, then 
$v_{n;i}(\lambda)=0$ for $i\in[n-N+d]=\pi_{\mathcal{I}_n}([d]\setminus\mathcal{I}_n)$. Noting that $\lambda_{n+1;1}=\lambda_{n+1;\kappa(n)}=\frac{N}{d}$ if $n\in[N-d+1,N-1]$, we have
\begin{align*}
v_{n;i}(\lambda)&=\sqrt{-\frac{\prod_{j\in[\kappa(n),d]} (\lambda_{n;i}-\lambda_{n+1;j})}{\prod_{j\in[\kappa(n),d]\setminus\{i\}}(\lambda_{n;i}-\lambda_{n;j})}}\\
&=\sqrt{-\frac{\prod_{j\in\{1\}\cup[n-N+d+2,d]} (\lambda_{n;i}-\lambda_{n+1;j})}{\prod_{j\in[\kappa(n),d]\setminus\{i\}}(\lambda_{n;i}-\lambda_{n;j})}}\\
&=\sqrt{-\frac{\prod_{j\in\mathcal{J}_n} (\lambda_{n;i}-\lambda_{n+1;j})}{\prod_{j\in\mathcal{I}_n\setminus\{i\}}(\lambda_{n;i}-\lambda_{n;j})}}\\
&=v_{n;\pi_{\mathcal{I}_n}(i)}
\end{align*}
if $n<N-1$, and $v_{N-1;d}(\lambda)=\sqrt{-(\lambda_{N-1;d}-\lambda_{N;d})}=\sqrt{-(\lambda_{N-1;d}-\lambda_{N;1})}=v_{n,1}$. The identity for $w_n(\lambda)$ follows similarly, and the identity for $W_n(\lambda)$ follows from the identities for $v_n(\lambda)$ and $w_n(\lambda)$. 

Now, in the inductive definition of $f_n$ in Theorem 7 of \cite{CFMPS12},
\[
U_{n+1}=U_n V_n \Pi_{\mathcal{I}_n}^T\begin{pmatrix}
W_n & {\bf 0} \\
{\bf 0} & I_{k(n)}
\end{pmatrix} \Pi_{\mathcal{J}_n}, 
\]
for some function $k(n)$, and the update used for our theorem is effectively
\[
U_{n+1} = U_n V_n \Pi_{\mathcal{I}_n}^T\begin{pmatrix}
W_n & {\bf 0} \\
{\bf 0} & I_{k(n)}
\end{pmatrix} \Pi_{\mathcal{I}_n}=U_n V_n \Pi_{\mathcal{I}_n}^T\begin{pmatrix}
W_n & {\bf 0} \\
{\bf 0} & I_{k(n)}
\end{pmatrix} \Pi_{\mathcal{J}_n}\Pi_{\mathcal{J}_n}^T\Pi_{\mathcal{I}_n}.
\]
For $n\leq N-d$, these definitions are identical because $\Pi_{\mathcal{J}_n}^T\Pi_{\mathcal{I}_n}=I_d$. For $n>N-d$, the application of $\Pi_{\mathcal{J}_n}^T\Pi_{\mathcal{I}_n}$ simply permutes the basis of eigenvectors for the eigenspace of the $(n+1)$th partial frame operator corresponding to the eigenvalue $\frac{N}{d}$. Since the next vector $f_{n+2}$ must be orthogonal to this subspace, it follows that the $f_{n+2}$ of both sequences coincides. This completes the proof.\end{proof}

\begin{lemma}\label{fulllift}
Given any $F\in\mathcal{F}_{N,d}^\mathbb{F}$ such that $\mu = \Lambda(F)\in\operatorname{int}(\Lambda_{N,d})$, there is a continuous map $\theta\colon\operatorname{int}(\Lambda_{N,d})\rightarrow \mathcal{F}_{N,d}^\mathbb{F}$ such that $\theta(\Lambda(F))=F$ and $\Lambda(\theta(\lambda))=\lambda$ for all $\lambda\in\operatorname{int}(\Lambda_{N,d})$.
\end{lemma}
\begin{proof}
Applying Theorem 7 of \cite{CFMPS12} to $F$, we obtain unitary matrices $U_1$ and $\{V_n\}_{n=1}^{N-1}$ that can be used to recover the individual columns of $F$ from $\mu$ and the matrices $v(\mu)$, $w(\mu)$, and $W_n(\mu)$ from Theorem \ref{interior-eigensteps}. 

Note that $v(\lambda)$, $w(\lambda)$ and $W_n(\lambda)$ are all continuous as functions of $\lambda\in\operatorname{int}(\Lambda_{N,d})$ because their entries are all radicals of positive-valued rational functions whose denominators are never zero on the domain of interest. Hence, the vector-valued functions defined by $f_1(\lambda)=u_1$, $U_1(\lambda)=U_1$, $f_{n+1}=U_n(\lambda) V_n v_n(\lambda)$, and $U_n(\lambda)=U_n(\lambda)V_n W_n(\lambda)$ are all continuous as well. Consequently, the matrix-valued function 
\[
\theta(\lambda)=\begin{pmatrix} f_1(\lambda) & f_2(\lambda) &\cdots & f_N(\lambda)\end{pmatrix}
\]
is continuous and takes values in $\mathcal{F}_{N,d}^\mathbb{F}$. Finally, the converse part of Theorem 7 of \cite{CFMPS12} ensures that $\theta(\mu)\in \mathcal{F}_{N,d}^\mathbb{F}$ satisfies $\Lambda(\theta(\mu))=\mu$.\end{proof}

\begin{lemma}\label{lift}
Suppose $F\in\mathcal{F}_{N,d}^\mathbb{F}$ with $N\geq d+2$, assume that $\mu = \Lambda(F)$ is in the interior of $\Lambda_{N,d}$, fix $\nu\in\Lambda_{N,d}$, and define the linear path $\ell\colon[0,1]\rightarrow\Lambda_{M,d}$ by $\ell(t)=(1-t)\mu+t\nu$. Then there is a continuous lifting (or frame path), $\widetilde{\ell}\colon[0,1]\rightarrow \mathcal{F}_{N,d}^\mathbb{F}$ such that $(\lambda\circ \widetilde{\ell})(t)=\ell(t)$ for all $t\in[0,1]$.
\end{lemma}

\begin{proof}
$F$ satisfies the hypothesis of Lemma \ref{fulllift}, so we may obtain a lift $\theta:\operatorname{int}(\Lambda_{N,d})\rightarrow\mathcal{F}_{N,d}^\mathbb{F}$ satisfying $\theta(\mu)=F$ and $\Lambda(\theta(\lambda))=\lambda$ for all $\lambda\in \operatorname{int}(\Lambda_{N,d})$. Thus, we may define the lift of $\ell$ by $\widetilde{\ell}=\theta\circ\ell:[0,1)\rightarrow\mathcal{F}_{N,d}^\mathbb{F}$ and verify that $\widetilde{\ell}(0)=F$. If $\nu$ is also on the interior, it is clear that the this path extends continuously to all of $[0,1]$. On the other hand, if $\nu$ lies on the boundary, then some interlacing inequalities become equalities ($\nu_{n;i}=\nu_{n+1;i}$ or $\nu_{n;i}=\nu_{n+1;i+1}$ for some $n$ and $i$) and the definitions of $v_n(\lambda)$, $w_n(\lambda)$, and $W_n(\lambda)$ in Lemma \ref{fulllift} appear to involve undefined quantities because of the terms in the denominators. First, we show that these problematic terms cancel along this path. 

We shall explicitly show that the problematic denominators cancel for the indices $n\in[d,N-d]$. When $n\in[d]$ or $n\in[N-d+1,N-1]$, the result follows from a similar argument, but one must take slightly more care with the indices under consideration. Now, fix $i\in[d]$ and note that the set of all $j$ such that $\nu_{n;i}=\nu_{n;j}$ ($\nu_{n+1;i}=\nu_{n+1;j}$ for $w(\lambda)$) forms an interval $[k_1,k_2]\subset [d]$ because these values are ordered. The interlacing conditions on $\nu$ then imply that $\nu_{n+1;k}=\nu_{n;i}$ for all $k\in[k_1,k_2-1]$. Setting
\[
h_j(t) = \left[(1-t)\mu_{n;i}+t\nu_{n;i}-(1-t)\mu_{n+1;j}-t\nu_{n+1;j}\right]
\]
and
\[
g_j(t) = \left[(1-t)\mu_{n;i}+t\nu_{n;i}-(1-t)\mu_{n;j}-t\nu_{n;j}\right]
\]
for all $j\in[d]$, note that 
\begin{itemize}
\item $h_k(t)=(1-t)(\mu_{n;i}-\mu_{n+1;k})$ for all $k\in[k_1,k_2-1]$;
\item $g_k(t)=(1-t)(\mu_{n;i}-\mu_{n;k})$ for all $k\in[k_1,k_2]$;
\item if $k\not\in [k_1,k_2]$, then $g_k(t)\not=0$ for all $t\in[0,1]$ since $\mu\in\operatorname{int}(\Lambda_{N,d})$ and $\nu_{n;i}\not=\nu_{n;k}$;
\item and $v_{n;i}(\ell(t))=\sqrt{-\prod_{j=1}^d h_j(t)/\prod_{j\not=i} g_j(t)}$ for $t\in[0,1)$.
\end{itemize}
Thus,
\begin{align*}
v_{n;i}(\ell(t)) &= \sqrt{-\frac{\prod_{j=1}^d h_j(t)}{\prod_{j\not=i} g_j(t)}}\\
&= \sqrt{-\frac{\prod_{j\in[k_1,k_2-1]} h_j(t)\prod_{j\not\in[k_1,k_2-1]} h_j(t)}{\prod_{j\in[k_1,k_2]\setminus\{i\}} g_j(t)\prod_{j\not\in[k_1,k_2]} g_j(t)}}\\
&= \sqrt{-\frac{\prod_{j\in[k_1,k_2-1]} (1-t)(\mu_{n;i}-\mu_{n+1;j})\prod_{j\not\in[k_1,k_2-1]} h_j(t)}{\prod_{j\in[k_1,k_2]\setminus\{i\}} (1-t)(\mu_{n;i}-\mu_{n;j})\prod_{j\not\in[k_1,k_2]} g_j(t)}}\\
&=\sqrt{-\frac{\prod_{j\in[k_1,k_2-1]} (\mu_{n;i}-\mu_{n+1;j})\prod_{j\not\in[k_1,k_2-1]} h_j(t)}{\prod_{j\in[k_1,k_2]\setminus\{i\}}(\mu_{n;i}-\mu_{n;j})\prod_{j\not\in[k_1,k_2]} g_j(t)}},
\end{align*}
and the denominator term $\prod_{j\in[k_1,k_2]\setminus\{i\}}(\mu_{n;i}-\mu_{n;j})\prod_{j\not\in[k_1,k_2]} g_j(t)$ is nonzero for all $t\in[0,1]$. 
Thus $v\circ\ell$ continuously extends to all of $[0,1]$. A similar argument extends $w\circ\ell$ continuously to all of $[0,1]$. 

If $\nu_{n+1;j}=\nu_{n;i}$ for some $i$ and $j$, then let the $[k_1,k_2]\subset[d]$ denote the set of all $j^\prime$ satisfying $\nu_{n;i}=\nu_{n;j^\prime}$ and let $[l_1,l_2]\subset[d]$ denote the set of all $j^\prime$ satisfying $\nu_{n+1;j}=\nu_{n+1;j^\prime}$, and note that $i\in[k_1,k_2]$ and $j\in[l_1,l_2]$. Setting
\[
b_k(t) = (1-t)\mu_{n+1;i}+t\nu_{n+1;i}-(1-t)\mu_{n;k}-t\nu_{n;k}
\]
and 
\[
a_k(t)= (1-t)\mu_{n+1;i}+t\nu_{n+1;i}-(1-t)\mu_{n+1;k}-t\nu_{n+1;k},
\]
we observe that
\begin{align*}
(W_n)_{j;i}(\ell(t)) &= \frac{v_{n;i}(\ell(t))w_{n;j}(\ell(t))}{(1-t)\mu_{n+1;j}+t\nu_{n+1;j}-(1-t)\mu_{n;i}-t\nu_{n;i}}\\
&= \frac{v_{n;i}(\ell(t))w_{n;j}(\ell(t))}{(1-t)(\mu_{n+1;j}-\mu_{n;i})}.
\end{align*}
This last numerator is
\begin{align*}
\sqrt{-\frac{\prod_{k\in[l_1,l_2]}h_k(t)\prod_{k\not\in[l_1,l_2]}h_k(t)}{\prod_{k\in[k_1,k_2]\setminus\{i\}}h_k(t)\prod_{k\not\in[k_1,k_2]}h_k(t)}}\sqrt{\frac{\prod_{k\in[k_1,k_2]}b_k(t)\prod_{k\not\in[k_1,k_2]}h_k(t)}{\prod_{k\in[l_1,l_2]\setminus\{j\}}a_k(t)\prod_{k\not\in[l_1,l_2]}a_k(t)}}.
\end{align*}
Similar to our previous considerations, we have that
\begin{itemize}
\item $h_k(t)=(1-t)(\mu_{n;i}-\mu_{n+1;k})$ for all $k\in[l_1,l_2]$ since $\nu_{n+1;k}=\nu_{n+1;j}=\nu_{n;i}$
\item $b_k(t)=(1-t)(\mu_{n+1;j}-\mu_{n;k})$ for all $k\in[k_1,k_2]$ since $\nu_{n;k}=\nu_{n;i}=\nu_{n+1;j}$
\end{itemize}
and therefore the $(1-t)$ term has an exponent of $(k_2-k_1)+(l_2-l_1)$ in the numerator under the radical, and an exponent of $(k_2-k_1)+(l_2-l_1)-2$ in the denominator under the radical. This leaves a factor of $(1-t)^2$ under the radical, which ultimately cancels with the $(1-t)$ term in the denominator of $(W_n)_{j;i}(\ell(t))$. We observe that all the denominator terms are now nonzero for all $t\in[0,1]$, so $W_n(\ell(t))$ continuously extends to all of $[0,1]$. 

Because we have continuous extensions of $v(\ell(t))$, $w(\ell(t))$, and the $W_n(\ell(t))$, our definition of $\theta\circ\ell$ extends continuously to $[0,1]$. Continuity of this extension then ensures that $\theta\circ\ell(t)\in\mathcal{F}_{N,d}^\mathbb{F}$ for all $t\in[0,1]$ since $\mathcal{F}_{N,d}$ is compact. Moreover, continuity gives us that $\Lambda(\theta(\ell(t))=\ell(t)$ for all $t\in[0,1]$. This completes our demonstration that $\widetilde{\ell}(t)$ is a continuous lifting of $\ell(t)$.\end{proof}

\section{Connectivity of $\mathcal{F}_{N,d}^\mathbb{F}$}\label{sec:connect}

We now prove the connectivity of real algebraic varieties of complex FUNTFs using the path lifting argument of the previous section.

\begin{proof}[Proof of Theorem \ref{cmplxconnect}]
The connectivity result is well known if $N=d$ or $N=d+1$, so we only consider the cases $N\geq d+2$ where the interior of $\Lambda_{N,d}$ is not empty. Let $F$ and $G$ belong to $\mathcal{F}_{N,d}^\mathbb{C}$. We consider two cases. First, suppose that $\Lambda(F)$ belongs to the interior of $\Lambda_{N,d}$. By Lemma \ref{lift}, we may continuously connect $F$ to a $G^\prime$ in $\mathcal{F}_{N,d}^\mathbb{C}$ with $\Lambda(G)=\Lambda(G^\prime)$. By Theorem 7 of \cite{CFMPS12}, we have that the only difference between $G$ and $G^\prime$ is the choice of $U_1$ and the $V_n$. However, since the unitary matrices are connected and products of connected sets are connected, there are continuous paths connecting $U_1^\prime $ to $U_1$ and each $V_n^\prime$ to $V_n$. These then induce a continuous path from $G^\prime$ to $G$. By traversing this path after the path provided by Lemma \ref{lift}, we successfully connect $F$ to $G$ by a continuous path. On the other hand, if $F$ and $G$ both belong to the boundary of $\Lambda_{N,d}$, we may choose $H$ from the interior of $\Lambda_{N,d}$ and connect $F$ to $H$ and $H$ to $G$ in the manner previously described. Traversing these paths in order produces the desired path.\end{proof}

In the complex case, the inverse image of a point of $\Lambda_{N,d}$ under the map $\lambda$ is connected because the $V_n$'s from Theorem 7 of \cite{CFMPS12} are block diagonal with unitary blocks and this set is connected (it is a product of unitary groups which are each individually connected). In the real case, this inverse image has at least two disjoint connected components. This makes the proof much more delicate, and we shall now require access to some simple continuous motions to circumvent the ostensible obstructions. Perhaps the most obvious motion that a FUNTF may undergo is the ``spinning'' of an orthogonal pair of frame elements inside of their span. The following lemma generalizes this kind of continuous motion.

\begin{lemma}\label{spinning}
Let $G=\{g_n\}_{n=1}^{N_1}\subset\mathbb{R}^d$ be a FUNTF for $W=\text{span} \{g_n\}_{n=1}^{N_1}$, and suppose $H=\{h_n\}_{n=1}^{N_2}\subset\mathbb{R}^d$ is such that $F=G\cup H\in\mathcal{F}_{N,d}^\mathbb{R}$. Let $U\in \mathcal{SO}(W)\hookrightarrow \mathcal{SO}(U)$ and set $G^\prime=\{Ug_n\}_{n=1}^{N_1}$. Then  $F^\prime=G^\prime\cup H\in\mathcal{F}_{N,d}^\mathbb{R}$, and there is a there is a continuous motion in $\mathcal{F}_{N,d}^\mathbb{R}$ connecting $F$ to $F^\prime$.
\end{lemma}

\begin{proof}
It is well known that $\mathcal{SO}(W)$ is connected, so there is continuous path $u:[0,1]\rightarrow \mathcal{SO}(W)$ such that $u(0)=I_d$ and $u(1)=U$. Defining $F(t) = \{u(t)g_n\}_{n=1}^{N_1}\cup H$, we observe that $\Vert u(t) g_n\Vert=\Vert g_n\Vert$ for all $t\in[0,1]$ and 
\begin{align*}
F(t)F(t)^\ast &= \sum_{n=1}^{N_1} u(t)g_n g_n^\ast u(t)^\ast + \sum_{n=1}^{N_2} h_n h_n^\ast\\
&= u(t)\left(\sum_{n=1}^{N_1} g_n g_n^\ast\right) u(t)^\ast + \sum_{n=1}^{N_2} h_n h_n^\ast\\
&= u(t) \left(\frac{N_1}{\text{dim}(W)} P_W \right)u(t)^\ast + \sum_{n=1}^{N_2} h_n h_n^\ast\\
&= \frac{N_1}{\text{dim}(W)} P_W+ \sum_{n=1}^{N_2} h_n h_n^\ast\\
&= \sum_{n=1}^{N_1} g_n g_n^\ast + \sum_{n=1}^{N_2} h_n h_n^\ast\\
& = FF^\ast = I_d.
\end{align*}
Therefore $F(t)\in\mathcal{F}_{N,d}^\mathbb{R}$ for all $t\in[0,1]$, and $F(t)$ is a continuous path with $F(0)=F$ and $F(1)=F^\prime$.\end{proof}

The next important motion involves ``swapping" vectors inside a FUNTF that is a union of two FUNTFs.

\begin{lemma}\label{swapping}
Let $G\in\mathcal{F}_{N_1,d}^\mathbb{R}$ and $H\in\mathcal{F}_{N_2,d}^\mathbb{R}$ and set $F=\begin{pmatrix} G & H\end{pmatrix}\in\mathcal{F}_{N,d}^\mathbb{R}$ where $N=N_1+N_2$. For any $N$ by $N$ permutation matrix $\Pi$, there is a continuous path from $F$ to $F\Pi$ in $\mathcal{F}_{N,d}^\mathbb{R}$. 
\end{lemma}

\begin{proof}
Without loss of generality, we simply show that such paths exist for permutation matrices arising from transpositions since any permutation is a product of transpositions and successive applications of transpositions may be obtained by concatenating paths.   

Suppose $\Pi$ represents the transposition of the $m$th column of $G=\begin{pmatrix}g_1&\cdots&g_{N_1}\end{pmatrix}$ with the $n$th column of $H\begin{pmatrix}h_1&\cdots&h_{N_1}\end{pmatrix}$ in $F$. Pick any $U\in \mathcal{SO}(d)$ such that $Ug_m=h_n$, note that the span of the vectors in $G$ is all of $\mathbb{R}^d$, and use Lemma \ref{spinning} to continuously connect to $\begin{pmatrix} UG & H\end{pmatrix}$. Now, the columns in $UG$ with indices in $[N_1]\setminus{m}$ form a FUNTF with the $n$th column of $H$. Applying Lemma \ref{spinning} to this FUNTF and the rotation $U^\ast$, we arrive at $F^\prime=F\Pi$.

If the pairs are from the same collection (say $G$), we first choose a third ``chaperone" vector from $H$, and then run three continuous paths swapping vectors from opposite collections. Let $g_{m_1}$, $g_{m_2}$, and $h_n$ denote the pair of vectors in $G$ and the chaperone in $H$, respectively. Furthermore, let $\Pi_1$ denote the permutation matrix which transposes $m_1$-th columns with $(N_1+n)$th columns by left multiplication, let $\Pi_2$ denote the permutation matrix which transposes the $m_2$-th columns with $(N_1+n)$th column by left multiplication, and note that $\Pi=\Pi_1\Pi_2\Pi_1$ is the permutation matrix which transposes the $m_1$-th column with $m_2$-th column by left multiplication. By the previous arguments, the application of $\Pi_1$, $\Pi_2$, and then $\Pi_1$ again may each be induced using continuous paths, and therefore concatenating these paths in the proper order may be used to induce a continuous path in $\mathcal{F}_{N,d}^\mathbb{R}$ from $F$ to $F\Pi$.\end{proof}

Finally, we need paths which induce negations of target vectors. 

\begin{lemma}\label{negating}
Let $G\in\mathcal{F}_{N_1,d}^\mathbb{R}$ and $H=\in\mathcal{F}_{N_2,d}^\mathbb{R}$ and set $F=\begin{pmatrix} G & H\end{pmatrix}\in\mathcal{F}_{N,d}^\mathbb{R}$ where $N=N_1+N_2$. For any $N$ by $N$ diagonal matrix $D$ with diagonal entries from the set $\{-1,1\}$, there is a continuous path from $F$ to $FD$ in $\mathcal{F}_{N,d}^\mathbb{R}$. 
\end{lemma}

\begin{proof}
We note that $D$ factors into a product of diagonal matrices with one diagonal entry equal to $-1$ and $N-1$ diagonal entries equal to $1$. Each of these matrices negates a single column of $F$, and since we may always concatenate continuous paths to get another continuous path, we therefore need only produce a continuous path that negates a single vector of $F$ to obtain the full result.

Without loss of generality, suppose we are attempting to negate $g_1\in G$. The first step in the construction of this path is to apply Lemma \ref{spinning} to position a chaperone $h_1\in H$ so that $g_1$ and $Uh_1$ are orthogonal. Then we may apply Lemma \ref{spinning} to simultaneously rotate $g_1$ and $Uh_1$ to $Uh_1$ and $-g_1$ respectively. All that is now required is to transpose the columns so that $g_1$ and $Uh_1$ arrive at $-g_1$ and $Uh_1$ respectively. This motion is obtained from Lemma \ref{swapping}. Finally, we use Lemma \ref{spinning} to return $H$ to its original position.\end{proof}

\begin{proof}[Proof of Theorem \ref{realconnect}]
We shall show the result using induction on $N$ inside of an induction on $d$, with the following induction structure:

\begin{itemize}
\item[(i)] For $d=2$, the result was shown in \cite{DS03} for all $N\geq 4$. 
\item[(ii)] For $d>2$, we have that connectivity of $\mathcal{F}_{N^\prime,d^\prime}^\mathbb{R}$ for all $N^\prime\geq d^\prime+2$ when $d^\prime <d$ implies connectivity of $\mathcal{F}_{N,d}^\mathbb{R}$ when $d+2\leq N<2d$.
\item[(iii)] For $d>2$, we prove connectivity independently for $N=2d,2d+1,$ and $2d+2$, and then show that connectivity in the case of $N^\prime$ implies connectivity in the case of $N^\prime+d+1$ (which becomes relevant for showing the cases $N=(d+2)+d+1=2d+3$ and beyond).
\end{itemize}

{\bf Case ($N=2d$): } Let $F$ and $G$ belong to $\mathcal{F}_{N,d}^\mathbb{R}$. Let $\mu$ denote the eigensteps for the frame consisting of two successive copies of the standard orthonormal basis. Using Lemma \ref{lift}, we connect both $F$ and $G$ to $F^\prime$ and $G^\prime$ in $\operatorname{int}(\Lambda_{N,d})$, and then connect $F^\prime$ and $G^\prime$ to $H$ and $H^\prime$ with $\Lambda(H)=\Lambda(H^\prime)=\mu$. This implies that $H$ and $H^\prime$ are both a union of two orthonormal bases. We shall now show any two unions of orthonormal bases may be continuously connected in $\mathcal{F}_{2d,d}^\mathbb{R}$.

Since orthonormal bases are permutation equivalent to positively oriented orthonormal bases, we apply Lemma \ref{swapping} to continuously connect our starting union of two successive orthonormal bases to a union of two successive, positively oriented orthonormal bases. Two applications of Lemma \ref{spinning} then connects this union of positively oriented orthonormal bases to the union of two successive standard orthonormal bases.

{\bf Case ($N=2d+1, 2d+2$):} Let $\mu$ denote any eigensteps for a frame $F$ such that $\{f_n\}_{n=1}^{d+1}\in\mathcal{F}_{d+1,d}^\mathbb{R}$ and $\{f_n\}_{n=d+1}^N\in\mathcal{F}_{N-d-1,d}^\mathbb{R}$. That is, $F$ is a union of two tight subframes. In particular, if $N=2d+1$, then $F$ is the union of a members of $\mathcal{F}_{d+1,d}^\mathbb{R}$ and an orthonormal basis. If $N=2d+2$, then $F$ is a union of two members of $\mathcal{F}_{d+1,d}^\mathbb{F}$. Just as in the $N=2d$ case, we now use Lemma \ref{lift} to connect any given $F$ and $G$ to $H$ and $H^\prime$ with eigensteps $\mu$ and note that $H,H^\prime\in\mathcal{F}_{d+1,d}^\mathbb{R}\times \mathcal{F}_{N-d-1,d}^\mathbb{R}$. We now show that $H$ and $H^\prime$ are connected by a path in $\mathcal{F}_{N,d}^\mathbb{R}$ to complete the demonstration of this case.

The central obstruction to connecting $F$ and $G$ in $\mathcal{F}_{d+1,d}^\mathbb{R}\times \mathcal{F}_{N-d-1,d}^\mathbb{R}$ is that $\mathcal{F}_{d+1,d}^\mathbb{R}$ has $2^d$ connected components as shown in \cite{DS03} (alternatively, one may observe that the these frames are the Naimark complements of $(d+1)$-length sequences of $1$s and $-1$s). Using Lemma \ref{negating}, we may continuously connect $F$ to any connected component in $\mathcal{F}_{d+1,d}^\mathbb{R}\times \mathcal{F}_{N-d-1,d}^\mathbb{R}$. Lemma \ref{spinning} then finalizes the connectivity result in these cases.

{\bf Case ($N$ implies $N+d+1$):} We simply connect the frame with $N+d+1$ members to a frame such that the first $d+1$ vectors form a member of $\mathcal{F}_{d+1,d}^\mathbb{R}$. The remaining $N$ vectors form a member of $\mathcal{F}_{N,d}^\mathbb{R}$ and hence we may generalize our previous arguments to connect to a frame whose first $d+1$ vectors form a particular member of $\mathcal{F}_{d+1,d}^\mathbb{R}$. Finally, the connectivity of $\mathcal{F}_{N,d}^\mathbb{R}$ implies that the last $N$ vectors in the frame may be connected to a particular member of $\mathcal{F}_{N,d}^\mathbb{R}$, and hence we have shown connectivity of $\mathcal{F}_{N+d+1,d}^\mathbb{R}$.

{\bf Case ($N^\prime\geq d^\prime+2\geq 4$ for all $d^\prime<d$ implies $d+2\leq N<2d$): } Suppose $N$ and $d$ satisfy $d+2\leq N<2d$, and let $N^\prime=N$ and $d^\prime = N-d$. Then $N^\prime>d^\prime+2$ and our induction step implies $\mathcal{F}_{N^\prime,d^\prime}^\mathbb{R}$ is connected. Thus, by Proposition \ref{equivconnect}, $\mathcal{F}_{N,d}^\mathbb{R}$ is connected.\end{proof}

\section{Real Algebraic Geometry of $\mathcal{F}_{N,d}^\mathbb{F}$}\label{sec:RAG}

A subset $V\subseteq\mathbb{R}^k$ is called a \emph{real algebraic variety} if there is a set of polynomials $\{p_i\}_{i\in I}\subseteq\mathbb{R}[x_1,...,x_k]$ such that $V=\{x\in\mathbb{R}^k:p_i(x)=0\text{ for every }i\in I\}$, where $\mathbb{R}[x_1,...,x_k]$ is the ring of polynomials in the variables $x_1,...,x_k$ with real coefficients. We also let $\langle \{p_i\}_{i\in I}\rangle$ denote the ideal in $\mathbb{R}[x_1,\ldots,x_k]$ generated by $\{p_i\}_{i\in I}$. Given an ideal $I\in\mathbb{R}[x_1,...,x_k]$, we let $V(I)$ denote the set of points $x$ for which $p(x)=0$ for all $p\in I$. Letting $\mathbb{R}[(x_{i,j})_{(i,j)\in[d]\times[N]}]$ denote the polynomials defined on the entries of real $N$ by $d$ matrices, we define the polynomials
\begin{align*}
p_i^\mathbb{R}(X)=&\frac{N}{d}-\sum_{j=1}^N x_{i,j}^2\\
q_j^\mathbb{R}(X)=&1-\sum_{i=1}^d x_{i,j}^2\\
r_{i,j}(X) = &\sum_{k=1}^N x_{k,i}x_{k,j},
\end{align*}
we set $\Pi_{N,d}^\mathbb{R}=\{p_i^\mathbb{R}\}_{i=1}^d\cup\{q_j^\mathbb{R}\}_{j=1}^{N-1}\cup\{r_{i,j}\}_{i<j\in[d]}$, we define the map $\pi_{N,d}^\mathbb{R}: M_{N,d}^\mathbb{R}\rightarrow\mathbb{R}^{Nd-(N-\frac{d}{2}-1)(d-1)}$ by
\[
\pi_{N,d}^\mathbb{R}(X) = \begin{pmatrix}
p_1^\mathbb{R}(X)\\
\vdots\\
p_d^\mathbb{R}(X)\\
q_1^\mathbb{R}(X)\\
\vdots\\
q_{N-1}^\mathbb{R}(X)\\
r_{1,2}(X)\\
\vdots \\
r_{1,d}(X)\\
\vdots\\
r_{d-1,d}(X)
\end{pmatrix},
\]
we use  $I_{N,d}^\mathbb{R}$ to denote ideal of polynomials generated by $\Pi_{N,d}^\mathbb{R}$, and note that $\mathcal{F}_{N,d}^\mathbb{R}=V(I_{N,d}^\mathbb{R})$. Letting 
\[
\mathbb{R}[(x_{i,j})_{(i,j)\in[d]\times[N]},(y_{i,j})_{(i,j)\in[d]\times[N]}]
\] denote the polynomials defined on the real and imaginary components of the entries in complex $N$ by $d$ matrices (where $\iota$ is chosen to satisfy $\iota^2=-1$ and complex entries are identified with $x_{i,j}+\iota y_{i,j}$), we define the polynomials
\begin{align*}
p_i^\mathbb{C}(X,Y)=&\frac{N}{d}-\sum_{j=1}^N x_{i,j}^2+y_{i,j}^2\\
q_j^\mathbb{C}(X,Y)=&1-\sum_{i=1}^d x_{i,j}^2+y_{i,j}^2\\
r_{i,j}^\Re(X,Y) =&\sum_{k=1}^N x_{k,i}x_{k,j}+y_{k,i}y_{k,j}\\
r_{i,j}^\Im(X,Y) =&\sum_{k=1}^N x_{k,i}y_{k,j}-y_{k,i}x_{k,j}.
\end{align*}
We define $\Pi_{N,d}^\mathbb{C}$, $\pi_{N,d}^\mathbb{C}$, and $I_{N,d}^\mathbb{C}$ analogously to the real case but using the above polynomials, and finally note that $\mathcal{F}_{N,d}^\mathbb{C}=V(I_{N,d}^\mathbb{C})$. Thus, we see that $\mathcal{F}_{N,d}^\mathbb{R}$ and $\mathcal{F}_{N,d}^\mathbb{C}$ are real algebraic varieties defined by quadratic polynomials. It should be noted that $q_N^\mathbb{F}$ is a linear combination of the polynomials $\{p_i^\mathbb{F}\}_{i=1}^d\cup\{q_j^\mathbb{F}\}_{j=1}^{N-1}$, so we have excluded it from the generating set.

A real semi-algebraic set consists of sets of the form 
\[
\bigcup_{i=1}^a\bigcap_{j=1}^b \{x\in\mathbb{R}^k: p_{i,j}(x)\ast_{i,j} 0\}
\]
where $p_{i,j}\in\mathbb{R}[x_1,\ldots,x_k]$ for $(i,j)\in[a]\times[b]$ and $\ast_{i,j}$ is ``$<$" or ``$=$".

\begin{proposition}\label{nod-semialgebraic}
The set of NOD frames in $\mathcal{F}_{N,d}^\mathbb{F}$ is semi-algebraic.
\end{proposition}

\begin{proof}
Note that $\ast_{i,j}$ in the definition of a semi-algebraic set may also be ``$\not=$'' because this set is encoded by $p_{i,j}(x)<0$ and $-p_{i,j}(x)<0$. For any nonempty proper subset $A\subset[N]$, let $A^c=[N]\setminus A$ denote the complement, and define the polynomials
\[
u_A^\mathbb{R}(X) = \sum_{n\in A} \sum_{m\in A^c} \left(\sum_{k=1}^d x_{k,n} x_{k,m}\right)^2
\] 
and
\small
\[
u_A^\mathbb{C}(X,Y) =  \sum_{n\in A} \sum_{m\in A^c} \left(\left(\sum_{k=1}^d x_{k,n} x_{k,m} + y_{k,n}y_{k,m}\right)^2 + \left(\sum_{k=1}^d x_{k,n} y_{k,m} - x_{k,n}y_{k,m}\right)^2\right).
\]\normalsize
Note that $u_A^\mathbb{R}(F)=0$ (and $u_A^\mathbb{C}(\Re(F),\Im(F))=0$ for the real and imaginary parts of $F\in\mathcal{F}_{N,d}^\mathbb{C}$) if and only if $\text{span}\{f_n\}_{n\in A}\perp\text{span}\{f_m\}_{m\in A^c}$. Therefore, the NOD frames of $\mathcal{F}_{N,d}^\mathbb{R}$ are exactly the intersection of the sets
\[
\bigcap_{i=1}^d \{X\in M_{N,d}^\mathbb{R}: p_i^\mathbb{R}(X)=0\},\: \bigcap_{j=1}^{N-1} \{X\in M_{N,d}^\mathbb{R}: q_i^\mathbb{R}(X)=0\},
\]
\[
\bigcap_{1\leq i<j\leq d} \{X\in M_{N,d}^\mathbb{R}: r_{i,j}^\mathbb{R}(X)=0\},\text{ and } \bigcap_{A\subset [N], A\not=\emptyset}  \{X\in M_{N,d}^\mathbb{R}: u_A^\mathbb{R}(X)\not=0\},
\]
and hence it is semi-algebraic. A similar definition involving $u_A^\mathbb{C}$ gives that the NOD frames of $\mathcal{F}_{N,d}^\mathbb{C}$ are also semi-algebraic.\end{proof}

A real semi-algebraic set $V$ induces an ideal $\mathcal{I}(V)\subset \mathbb{R}[x_1,\ldots,x_k]$ which consists of all polynomials that vanish on $V$. If $I\subset\mathbb{R}[x_1,\ldots,x_k]$ is an ideal, it is important to note that $\mathcal{I}(V(I))=\sqrt{I}$, where $\sqrt{I}$ is the radical of the ideal $I$, and $\sqrt{I}\not = I$ in general. The \emph{dimension} of a real semi-algebraic set $V$ (denoted $\text{dim}(V)$) is defined to be the dimension of the ring of polynomials $\mathcal{P}(V)=\mathbb{R}[x_1,\ldots,x_k]/\mathcal{I}(V)$, which is equal to the maximal length of chains of prime ideals in $\mathcal{P}(V)$. While the dimension of a real semi-algebraic set may be difficult to compute directly from this defintion, the following proposition shows that the dimension is the usual dimension when the real semi-algebraic is also a manifold.

\begin{proposition}[c.f. Proposition 2.8.14 of \cite{BCR13}]
Let $V\subset\mathbb{R}^k$ be a semi-algebraic set which is a $C^\infty$ submanifold of dimension $d$ in $\mathbb{R}^k$. Then the dimension of $V$ as a semi-algebraic set is also $d$.
\end{proposition}

Combining this proposition, Proposition \ref{nod-semialgebraic}, and Corollary 4.9 of \cite{DS03}, we obtain the following proposition.

\begin{proposition}\label{nod-dimension}
The subset of NOD frames of $\mathcal{F}_{N,d}^\mathbb{R}$ has dimension $(N-\frac{d}{2}-1)(d-1)$ as a real semi-algebraic set, and the subset of NOD frames in $\mathcal{F}_{N,d}^\mathbb{C}$ has dimension $2d(N-d)+d^2-N+1$ as a real semi-algebraic set.
\end{proposition}

In order to define nonsingular points on $\mathcal{F}_{N,d}^\mathbb{F}$, we actually need their dimension as real algebraic varieties. The first step is to demonstrate that the NOD frames are actually dense in $\mathcal{F}_{N,d}^\mathbb{F}$.

\begin{proposition}\label{nod-dense}
The subset of NOD frames in $\mathcal{F}^{\mathbb{F}}_{N,d}$ is dense in $\mathcal{F}^{\mathbb{F}}_{N,d}$.
\end{proposition}

\begin{proof}
Note that a frame in $F\in\mathcal{F}_{N,d}^\mathbb{F}$ is either OD or NOD, so we need only show that there is a NOD frame arbitrarily close to any OD frame. Given an OD frame $F\in\mathcal{F}_{N,d}^\mathbb{F}$, partition it into maximal NOD subsets. An orthonormal set $\{u_i\}_{i=1}^k$ may be obtained by choosing a single vector from each of maximal NOD subsets. Note that for any orthonormal set $\{u_i\}_{i=1}^k$ and any $\varepsilon>0$, there is another orthonormal set $\{\widetilde{u}_i\}_{i=1}^k\subset\text{span}\{u_i\}_{i=1}^k$ such that $\sum_{i=1}^k\Vert u_i-\widetilde{u}_i\Vert^2<\varepsilon^2$, and such that $\langle u_i,\widetilde{u}_j\rangle\not=0$ for all $i,j=1,\ldots,k$. Replacing the vectors $\{u_i\}_{i=1}^k$ with $\{\widetilde{u}_i\}_{i=1}^k$ in our original frame to obtain $\widetilde{F}$, we see that $\widetilde{F}\in\mathcal{F}_{N,d}^\mathbb{F}$ by Lemma \ref{spinning} and we claim that $\widetilde{F}$ is also a NOD frame for $\varepsilon$ sufficiently small. To prove this last step, note that the correlation network of $\widetilde{F}$ retains all the edges from the correlation network of the old frame since $\Vert u_i-\widetilde{u}_i\Vert<\varepsilon$ for $\varepsilon$ small implies $\langle f_k,\widetilde{u}_i\rangle\not=0$ if $\langle f_k,u_i\rangle\not=0$ whenever $f_k$ comes from a finite list. Then the fact that $\langle u_i,\widetilde{u}_j\rangle\not=0$ for $i\not=j$ implies that $\langle f_k,\widetilde{u}_j\rangle\not=0$ whenever $\langle f_k,u_i\rangle\not=0$. Thus, $\gamma(\widetilde{F})$ retains all the connected components of $\gamma(F)$, but also includes edges between all these connected components. Consequently, $\gamma(\widetilde{F})$ is connected and hence $\widetilde{F}$ is NOD. Finally, note that $\Vert F-\widetilde{F}\Vert<\varepsilon$ by construction.\end{proof}

Proposition \ref{nod-dense} gives us that the closure of the subset of NOD frames in $\mathcal{F}_{N,d}^\mathbb{F}$ is all of $\mathcal{F}_{N,d}^\mathbb{F}$. Proposition 2.8.8 of \cite{BCR13} then gives us the dimension of $\mathcal{F}_{N,d}^\mathbb{F}$ as a real-algebraic set has the same dimension as the real-algebraic set of NOD frames in $\mathcal{F}_{N,d}^\mathbb{F}$.

\begin{proposition}\label{funtf-dimension}
As real algebraic sets, the dimensions of $\mathcal{F}_{N,d}^\mathbb{R}$ and $\mathcal{F}_{N,d}^\mathbb{C}$ are $(N-\frac{d}{2}-1)(d-1)$ and $2d(N-d)+d^2-N+1$, respectively.
\end{proposition}

The definition of a nonsingular point of a real algebraic set requires more algebraic background than is necessary for the rest of this paper. The interested reader is referred to Definition 3.3.9 of \cite{BCR13}. For our purposes, we only need the following characterization.

\begin{proposition}[c.f. Proposition 3.3.10 of \cite{BCR13}]
Let $V\subset \mathbb{R}^k$ be a real algebraic variety of dimension $d$. Then $x\in V$ is nonsingular in dimension $d$ if and only if
there exist $k-d$ polynomials $p_1,\ldots,p_{k-d}\in \mathcal{I}(V)$ and an open neighborhood $U$ of $x$ in $\mathbb{R}^k$ such that
\begin{enumerate}
\item $V\cap U = V(\langle p_1,\ldots,p_{k-d}\rangle)\cap U$
\item and the Jacobian
\[
\begin{pmatrix}
\frac{\partial p_1}{dx_1}(x) & \cdots & \frac{\partial p_1}{dx_k}(x)\\
\vdots & \ddots & \vdots \\
\frac{\partial p_{k-d}}{dx_1}(x) & \cdots & \frac{\partial p_{k-d}}{dx_k}(x)
\end{pmatrix}
\]
has rank $k-d$.
\end{enumerate}
\end{proposition}

We now claim that any NOD frame $F\in\mathcal{F}_{N,d}^\mathbb{F}$ is nonsingular. Since the FUNTF varieties are locally manifolds around NOD frames by Corollary 4.9 of \cite{DS03}, there is an open neighborhood $U$ of $F$ in $M_{N,d}^\mathbb{F}$ such that 
\[
\mathcal{F}_{N,d}^\mathbb{R}\cap U = V(I_{N,d}^\mathbb{F})\cap U,
\] 
and by the characterization of the tangent spaces of $\mathcal{F}_{N,d}^\mathbb{R}$ obtained in \cite{S2007}, we know that the null space of the Jacobian of $\pi_{N,d}^\mathbb{F}$ has dimension $\text{dim}(\mathcal{F}_{N,d}^\mathbb{F})$.

\begin{proposition}\label{OD}
The NOD frames in $\mathcal{F}_{N,d}^\mathbb{F}$ are nonsingular in dimension $\text{dim}(\mathcal{F}_{N,d}^\mathbb{F})$, and hence the singular points of $\mathcal{F}_{N,d}^\mathbb{F}$ are contained in set of OD frames of $\mathcal{F}_{N,d}^\mathbb{F}$.
\end{proposition}

Proposition \ref{OD} brings up the possibility that orthodecomposability may not fully characterize the algebraic singularities of $\mathcal{F}_{N,d}^\mathbb{F}$, which provides an interesting problem for future consideration.

\begin{problem}
Either show that $F\in \mathcal{F}_{N,d}^\mathbb{F}$ is a singularity of $\mathcal{F}_{N,d}^\mathbb{F}$ if and only if  $F$ is OD, or exhibit an OD frame in $\mathcal{F}_{N,d}^\mathbb{F}$ is not a sigularity of $\mathcal{F}_{N,d}^\mathbb{F}$.
\end{problem}

By defining closed sets to be real algebraic varieties, we get a topology on $\mathbb{F}^k$ called the \emph{real Zariski topology} (note that this is different from the usual Zariski topology on $\mathbb{C}^k$). For a subset $V\subseteq\mathbb{F}^k$ we use the notation $\mathcal{Z}(V)$ to denote the closure of $V$ in this topology, that is, $\mathcal{Z}(V)$ is the smallest variety containing $V$. We will also use the real Zariski topology of a real algebraic variety $V\subseteq\mathbb{F}^k$ to mean the subspace topology of the real Zariski topology of $\mathbb{F}^k$. Note that any set which is closed in the real Zariski topology is also closed in the standard topology, but the converse of this is far from true.

A variety $V\subseteq\mathbb{F}^k$ is called \emph{irreducible} if we cannot write $V=V_1\cup V_2$ where $V_1$ and $V_2$ are proper subvarieties of $V$. A variety is called \emph{nonsingular} if it has no singular points. If a variety is reducible and path-connected, then any point in the intersection of two irreducible components is not nonsingular. Therefore, if a variety is path-connected and nonsingular, then it must be irreducible. When $N$ and $d$ are not relatively prime, it is a simple exercise to construct OD frames in $\mathcal{F}^{\mathbb{F}}_{N,d}$, and the possibility that these are singular points means that irreducibility does not follow immediately from connectedness. In our case, Proposition \ref{irredprop} gives us a way forward.

\begin{proposition}\label{irredprop}
Suppose $V$ is a real algebraic variety such that
\begin{itemize}
\item[(i)] the set of nonsingular points of $V$ is path-connected, and
\item[(ii)] the set of nonsingular points is dense in $V$ (in the standard topology).
\end{itemize}
Then $V$ is an irreducible real algebraic variety.
\end{proposition}
\begin{proof}
Let $V_0$ denote the set of nonsingular points of $V$.  We first claim that (i) implies $\mathcal{Z}(V_0)$ is irreducible.  To see this, suppose to the contrary that there are two subvarieties of $\mathcal{Z}(V_0)$, say $V_1$ and $V_2$, such that $V_1\cup V_2=\mathcal{Z}(V_0)$ and there exists $x\in V_0\cap V_1$ and $y\in V_0\cap V_2$.  Then any path in $V_0$ connecting $x$ and $y$ must pass through $V_1\cap V_2$ (since $\mathcal{Z}(V_0)\setminus(V_1\cap V_2)=(\mathcal{Z}(V_0)\setminus V_1)\cup(\mathcal{Z}(V_0)\setminus V_2)$ is disconnected). Overall, we have that $V_1\cap V_2$ intersects $V_0$ nontrivially, but this contradicts the fact that components must intersect at singular points; this fact follows from Proposition 3.3.10 of \cite{BCR13}.

Next, we apply (ii) to get $\overline{V_0}=V$, where bar denotes closure in the standard topology. Since any Zariski closed set is also closed in the standard topology, we further have $V=\overline{V_0}\subseteq\mathcal{Z}(V_0)$. Moreover, since $V_0\subseteq V$ and $V$ is Zariski closed, the reverse containment also holds: $\mathcal{Z}(V_0)\subseteq V$. As such, $V=\mathcal{Z}(V_0)$ and so $V$ is irreducible by the previous paragraph.\end{proof}

The converse of Proposition~\ref{irredprop} is false (see Figure~\ref{figure1} for counterexamples). In our situation, we note that the argument of Proposition \ref{nod-dense} indicates that, for any OD frame $F\in\mathcal{F}_{N,d}^\mathbb{F}$, there is a continuous path $\gamma:[0,1]\rightarrow \mathcal{F}_{N,d}^\mathbb{F}$ such that $\gamma(t)$ is a NOD frame for any $t\in[0,1)$ and $\gamma(1)=F$. Moreover, the NOD frames are dense in $\mathcal{F}_{N,d}^\mathbb{F}$ and are a subset of the nonsingular points of $\mathcal{F}_{N,d}^\mathbb{F}$, and hence the nonsingular points of $\mathcal{F}_{N,d}^\mathbb{F}$ are dense in $\mathcal{F}_{N,d}^\mathbb{F}$. Consequently, the following lemma then follows from Proposition \ref{irredprop}.

\begin{lemma}\label{nod-irreducible}
If the NOD frames in $\mathcal{F}_{N,d}^\mathbb{F}$ form a connected set, then $\mathcal{F}_{N,d}^\mathbb{F}$ is an irreducible real algebraic variety.
\end{lemma}

\begin{figure}[h]
\label{figure1}
\subfigure[$y^2=x^2(x+1)$]{
\includegraphics[scale=0.6]{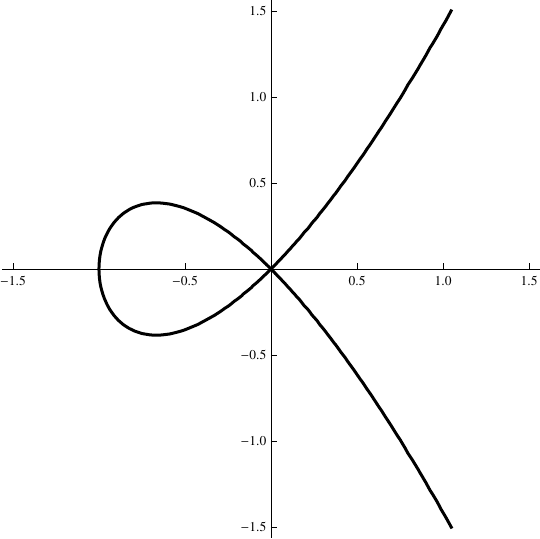}
}
\subfigure[$x^4+y^4=x^2+y^2$]{
\includegraphics[scale=0.6]{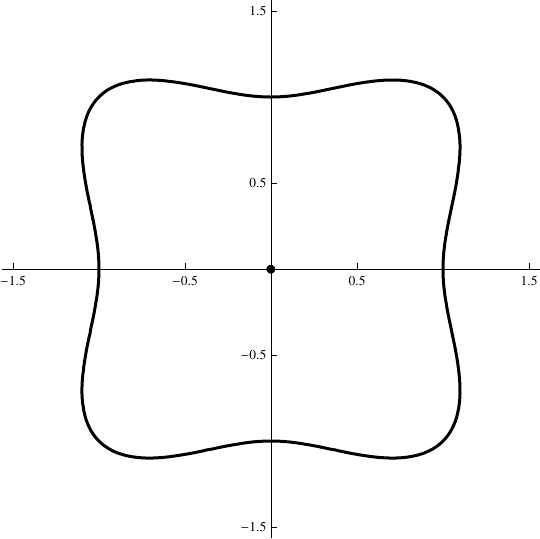}
}
\caption{Real algebraic counterexamples to the converse of Proposition~\ref{irredprop}. (a) An irreducible variety which violates (i); removing the singularity at the origin would disconnect the nonsingular points into three pieces. (b) An irreducible variety which violates (ii); the nonsingular points form a curve which is not dense in the variety, since the variety also contains an isolated point (a singularity).
}
\end{figure}

Given a subset of Euclidean space $V\subseteq\mathbb{R}^k$, we say $W\subseteq V$ is \emph{generic} (in $V$) if $W$ contains an open and dense set in the topology induced on $V$ by the standard topology on $\mathbb{R}^k$. If $V$ is an irreducible variety over the complex numbers then any Zariski-open subset of $V$ is generic, however this is not necessarily the case for irreducible real algebaric varieties. Indeed, the variety shown in Figure~\ref{figure1}(b) is an example of an irreducible variety that contains a Zariski-open set that is not dense in the standard topology; namely, the complement of the origin. Nevertheless, using the same hypotheses from Proposition~\ref{irredprop}, we can demonstrate the following:

\begin{proposition}\label{realgeneric}
Let $V$ be a real algebraic variety such that the nonsingular points of $V$ form a dense connected subset. If $U$ is another real algebraic variety, then either $V\setminus U$ is empty or generic in $V$.
\end{proposition}

\begin{proof}
Equivalently, we show that either $V\subset U$ or $V\cap U$ is nowhere dense in $V$ in the topology on $V$ induced by the standard topology. If $V\cap U$ is nowhere dense in $V$ in the relative topology, then the statement holds, so suppose that $V\cap U$ is not nowhere dense in $V$ in the relative toplogy. Because $V\cap U$ is closed and not nowhere dense, this means that $V\cap U$ contains a nonempty open subset $Q\subset V$ in this relative topology. Now, let $\mathcal{U}$ denote an open cover of the nonsingular points of $V$ so that for each $W\in\mathcal{U}$, there is an analytic coordinate patch on $W$. If $Q$ does not intersect some member $W$ of $\mathcal{U}$, then $Q$ is entirely contained in the singular points of $V$, which contradicts the hypothesis that the nonsingular points of $V$ form a dense subset of $V$. Thus, $Q$ intersects some $W\in \mathcal{U}$, and hence there is a connected, nonempty open subset $R\subset Q\cap W$. 

Now, there are real polynomials $f$ and $g$ such that
\[
V = \{x: f(x) =0\}\text{ and } U = \{x:g(x)=0\}
\]
since $V$ and $U$ are real algebraic varieties. Thus, $f$ and $g$ must coincide on $R$. Because $W$ admits an analytic coordinate system $\phi:Q\rightarrow W$ for $Q$ some nonempty open subset of a Euclidean space, we have that
\[
f\circ\phi - g\circ \phi = 0
\]
on $\phi^{-1}(R)$ as a multivariate analytic function. We claim that $f\circ\phi-g\circ\phi=0$ on all of $\phi^{-1}(W)$. This is simply a consequence of the Identity Principle for single-variable analytic functions. That is, suppose $h\colon\mathbb{C}^k\rightarrow\mathbb{C}$ is analytic and that $h(x)=0$ for all $x\in\mathbb{R}^k$ with $\Vert x\Vert < r$. Then by fixing $x_2,\ldots,x_k$ with $x_2^2+\cdots+x_k^2<r^2$, we have that
\[
h_1(x_1) = h(x_1,x_2,\ldots,x_k)
\]
is a one-dimensional analytic function that vanishes on a sequence of points having an accumulation point in $\mathbb{C}$. Consequently, $h_1(x_1) = 0$ for all $x_1\in\mathbb{C}$, and hence
\[
h(x_1,\ldots,x_k)=h(x_2,\ldots,x_k).
\]
Now, since $h(x_2,\ldots,x_k)$ still vanishes on a ball contained in $\mathbb{R}^{k-1}$, we may use induction to see that $h=0$ on all of $\mathbb{C}^k$. 

We then have that $f=g$ on all of $W$. If $W^\prime\in \mathcal{U}$ intersects $W$ nontrivially, then the same reasoning as above shows that $f=g$ on $W^\prime$. Now, let $A$ denote the union of all $W^\prime\in \mathcal{U}$ such that $f=g$ on $W^\prime$ and let $B$ denote the union of all open sets $W^\prime\in\mathcal{U}$ such that $f\not = g$ on $W^\prime$. Then $A$ and $B$ are both open, and if $A$ and $B$ intersect nontrivially, there is a nontrivial intersection between some $W^\prime$ and $W^{\prime\prime}$ in $\mathcal{U}$ such that $f=g$ on $W^\prime$ and $f\not = g$ on $W^{\prime\prime}$. This is a contradiction by our above observation, so we see that the set of nonsingular points of $V$ coincides with the disjoint union of open $A$ and $B$, which, by connectedness of $V$ and nonemptyness of $A$, implies that $B$ is empty and hence $f=g$ on all nonsingular points of $V$.

Since $U$ is closed, the above reasoning immediately implies that $V\subset U$.\end{proof}

Overall, if the nonsingular points of a real algebraic variety form a dense connected subset, then the nonempty Zariski-open subsets of that variety are generic, as desired. It should be noted that the above proposition employs the topological definition of connectivity, but Corollaries \ref{cmplxNODconnect} and \ref{realNODconnect} say that our set of nonsingular points form a path-connected set. However, it is clear that the nonsingular points form an analytic manifold, and hence these two definitions of connectivity coincide. 

\section{Connectivity of the subset of NOD frames in $\mathcal{F}_{N,d}^\mathbb{F}$}\label{sec:nod-connect}

In this section, we refine the results of the previous section to show that, given two NOD frames $F,G\in\mathcal{F}_{N,d}^\mathbb{F}$, there is a continuous path connecting $F$ and $G$ in $\mathcal{F}_{N,d}^\mathbb{F}$ that does not pass through the OD frames. 

To begin with, we need some lemmas that reveal the structure of the OD frames. Our first lemma essentially allows us to know that a frame is NOD if we can extract a NOD basis from the frame.

\begin{lemma}[Proposition~4.2 in~\cite{S2012}]\label{basisODC}
A frame $F$ is OD if and only if every basis contained in $F$ is OD.
\end{lemma}

This next lemma tells us that the eigensteps of an OD frame must be on the boundary of $\Lambda_{N,d}$. Thus, when we use Lemma \ref{lift} to lift paths through the interior of $\Lambda_{N,d}$, we know that the path never crosses an OD frame.

\begin{lemma}\label{boundaryOD}
Suppose $F\in \mathcal{F}_{N,d}^{\mathbb{F}}$ is OD then $\Lambda(F)\in \partial \Lambda_{N,d}$.
\end{lemma}

\begin{proof}
Since $F$ is OD, there is an index $k>1$ such that $f_k$ is orthogonal to $f_i$ for each $i<k$. Thus, the nonzero values of $\lambda_k(F)$ consists of a $1$ and all of the nonzero values of $\lambda_{k-1}(F)$. Since the largest nonzero value of $\lambda_{k-1}$ is at least 1, this means that $\lambda_{k-1; 1}(F) = \lambda_{k;1}(F)$ and hence $\Lambda(F)$ is on the boundary of $\Lambda_{N,d}$.\end{proof}

Our final lemma for this section tells us that if we reorder any NOD frame so that the first $d$ vectors of the frame form a NOD basis, then there are no OD frames that map to the same eigensteps. This means that when the $V_n$'s are being continuously diagonalized in Lemma \ref{lift}, the path avoids the OD frames.

\begin{lemma}\label{reorderODC}
Let $F=\{f_n\}_{n=1}^N\in\mathcal{F}_{N,d}^\mathbb{F}$.
\begin{itemize}
\item[(i)] If the first $d$ vectors in $F$, $\{f_n\}_{n=1}^d$, form a NOD basis, then $\Lambda(F)$ is not in the image of the OD frames under the eigensteps map.

\item[(ii)] If $F$ is NOD, then there is a permutation $\sigma$ such that $\{f_{\sigma(n)}\}_{n=1}^d$ is a NOD basis and hence $F^\prime=\{f_{\sigma(n)}\}_{n=1}^N$ has eigensteps which are not in the image of the OD frames under the eigensteps map.
\end{itemize}

\end{lemma}

\begin{proof}
First, we prove part (i). For the eigensteps of an OD basis $\{g_n\}_{n=1}^N$, there will be a $k>1$ such that the nonzero values of $\lambda_k(\{g_n\}_{n=1}^N)$ consist of the nonzero values of $\lambda_{k-1}(\{g_n\}_{n=1}^d)$ together with 1. Consequently, there are no OD frames with eigensteps $\Lambda(\{f_n\}_{n=1}^N)$.

To prove part (ii), we first identify a NOD basis from $F$, $B=\{b_n\}_{n=1}^d$. Let $b_1=f_1$ and set $n_1=1$. Inductively, we now choose $b_k=f_{n_k}$ for $d\geq k>1$ so that $f_{n_k}$ is not in $\text{span}^\perp \{b_n\}_{n=1}^{k-1}$ or $\text{span} \{b_n\}_{n=1}^{k-1}$, and $n_k$ is not in $\{n_1,\ldots,n_{k-1}\}$. We claim that there is always such an $f_{n_k}$. Suppose there is not. Then all $f_n$ such that $n\not\in\{n_1,\ldots,n_{k-1}\}$ are either in $\text{span}^\perp \{b_n\}_{n=1}^{k-1}$ or $\text{span} \{b_n\}_{n=1}^{k-1}$, and since neither of these spaces will be empty, this implies that $F$ is necessarily OD, which contradicts our hypothesis. As such, we may take $\sigma$ to be any permutation sending $k$ to $n_k$ for all $k=1,\ldots,d$.\end{proof}

Now that we have these lemmas, it is fairly straightforward to prove the refined connectivity result for all of the cases where $N\not = 2d$.

\begin{theorem}\label{non-orthconnect}
Suppose $N\geq d+2\geq 4$ and $N\not = 2d$. Then the NOD members of $\mathcal{F}_{N,d}^\mathbb{F}$ form a path-connected set.
\end{theorem}

\begin{proof}
The structure of this proof is similar to that of Theorem \ref{realconnect}. We proceed with the exact induction structure, but we now take care to construct paths that do not pass through OD frames. Additionally, the application of Proposition \ref{equivconnect} in Theorem \ref{realconnect} is replaced with the application of Proposition \ref{nodequivconnect}, and the case $N\geq 2d+1$ only needs the connectivity result in Theorem \ref{realconnect}.

The first observation we must make is that if there are no OD frames with eigensteps $\lambda$ or $\mu$, then the path connecting $F$ and $G$ with eigensteps $\lambda$ and $\mu$ provided in Theorem \ref{realconnect} never passes through any OD frames. The argument supporting this statement occurs in two steps. First, two frames $F$ and $G$ in $\mathcal{F}_{N,d}^\mathbb{F}$ with $N>2d$ are connected to frames $H$ and $H^\prime$ with eigensteps $\nu$ that ensure the first $d+1$ members of  $H$ and $H^\prime$ each form a a member of $\mathcal{F}_{d+1,d}^\mathbb{R}$. Since there are no OD frames with eigensteps $\nu$ (Lemma \ref{basisODC} gives this since the first $d$ vectors of a member of $\mathcal{F}_{d+1,d}^\mathbb{R}$ form a NOD basis), the path provided by Lemma \ref{lift} does not pass through eigensteps that are in the image of the OD frames under the eigensteps map. Finally, note that each frame in the path from $H$ to $H^\prime$ constructed in Theorem \ref{realconnect} always contains $d$ vectors of a a member of $\mathcal{F}_{d+1,d}^\mathbb{R}$, which is necessarily a nonorthdecomposable basis and hence Lemma \ref{basisODC} ensures that the full frame is not OD. 

The only remaining issue to deal with is that either $F$ or $G$ may have eigensteps in the image of the OD frames under the eigensteps map, so the continuous diagonalization of the $V_n$'s might cross an OD frame. Let us suppose that there is an OD frame with the same eigensteps as $F$. Since $F$ is NOD Lemma \ref{reorderODC} provides us with a permutation $\sigma$ such that $F^\prime=\{f_{\sigma(n)}\}_{n=1}^N$ has that $\Lambda(F^\prime)$ is not in the image of the OD frames under the eigensteps map. Therefore we may use the reasoning in the preceding paragraph to connect $F^\prime$ to a frame $H$ such that $\{h_n\}_{n=1}^{d+1}$ is a member of $\mathcal{F}_{d+1,d}^\mathbb{R}$, and in such a way that no frame along this path is OD. Let $\ell(t)=\{f_n^\prime(t)\}_{n=1}^N$ denote this path and note that $\{f_{\sigma^{-1}(n)}^\prime(t)\}_{n=1}^N$ is also a path through NOD frames which takes our original $F$ to a frame which contains a member of $\mathcal{F}_{d+1,d}^\mathbb{F}$. Let $H$ denote this frame.

Now, we use Lemma \ref{swapping} to rearrange $H$ so that the first $d+1$ vectors in $H$ form a member of $\mathcal{F}_{d+1,d}^\mathbb{F}$. Note that the intermediate frames in the path produced by Lemma \ref{swapping} will always contain at least $d$ vectors from a member of $\mathcal{F}_{d+1,d}^\mathbb{F}$, and hence all these intervening frames along the path are NOD.

Finally, we finish the proof by observing that our paths produced in Theorem \ref{realconnect} to connect frames containing a member of $\mathcal{F}_{d+1,d}^\mathbb{F}$ as the first $d+1$ vectors always consist of frames which have at least $d$ members of a frame from $\mathcal{F}_{d+1,d}^\mathbb{F}$. Thus, the final paths connecting $F$ and $G$ also avoid the OD frames, and the proof is complete.\end{proof}

Using Lemma \ref{genericunitaries}, it is not difficult to show that $\mathcal{F}_{2d,d}^\mathbb{C}$ is also path connected.

\begin{lemma}[Lemma~2.2 in~\cite{CCPW13}]\label{genericunitaries}
Two generic orthonormal bases are full spark.
\end{lemma}

\begin{theorem}
The NOD members of $\mathcal{F}^{\mathbb{C}}_{2d,d}$ are path-connected.
\end{theorem}

\begin{proof}
As in the proof of Theorem \ref{non-orthconnect} if both $F$ and $G$ have eigensteps which are not in the image of the OD frames, then the paths constructed in Section \ref{sec:connect} will not pass through OD frames. Without loss of generality, we can now assume that there is an OD frame that has the same eigensteps as $F$. Since $F$ is not OD, Lemma \ref{reorderODC} provides us with a permutation $\sigma$ such that $F^\prime=\{f_{\sigma(n)}\}_{n=1}^N$ has that $\Lambda(F^\prime)$ is not in the image of the OD frames under the eigensteps map and note that this permutation satisfies $\sigma(1)=1$. By Lemma \ref{genericunitaries}, we can choose two orthonormal bases, say $\{e_i\}_{i=1}^d$ and $\{u_i\}_{i=1}^d$ such that $\{e_i\}_{i=1}^d\cup\{u_i\}_{i=1}^d$ is full spark. Reorder these vectors into a frame $H=\{h_n\}_{n=1}^{2d}$ so that $h_1=e_1,h_2=u_1$, and $h_{\sigma^{-1}(2)}=u_2$. Observe that $\{h_n\}_{n=1}^d$ is a NOD basis so $\Lambda(H)$ is not in the image of the OD frames by Lemma \ref{reorderODC}. Now we can connect $F'$ to $H$ so that the path does not go through any OD frames. Next apply $\sigma^{-1}$ to this path to get a path from $F$ to $\sigma^{-1}(H)$ which does not pass through any OD frame. Note that $\Lambda(\sigma^{-1}(H))$ is not in the image of the OD frames since $h_{\sigma(1)}=e_1$ and $h_{\sigma(2)}=u_2$ and $H$ is full spark so $\{h_{\sigma(n)}\}_{n=1}^d$ is a NOD basis. We have shown that any NOD frame in $\mathcal{F}^{\mathbb{C}}_{2d,d}$ can be connected to a frame whose eigensteps are not in the image of the OD frames. It follows that we can connect any two NOD frames in this set without passing through an OD frame.\end{proof}

The above result does not apply to the real case when $N=2d$. In the process of connecting our $F^\prime$ to $H$, in the real case we may only ensure that we can connect $F^\prime$ to a frame sharing the same eigensteps as $H$. Additionally, the path constructed in Theorem \ref{realconnect} for the case $N=2d$ requires that we align frame vectors directly on top of each other. Once this alignment occurs, we have reached an OD frame. This makes the proof for $\mathcal{F}_{2d,d}^\mathbb{R}$ much more involved.

In order to show that we can connect two NOD frames in $\mathcal{F}_{2d,d}^\mathbb{R}$ without passing through an OD frame, we shall still use the unions of two orthonormal bases as the central nexus of our paths. As in Theorem \ref{realconnect}, once we get to a single union of two orthonormal bases which constitutes a NOD frame (a nontrivial task), we just need to show that we can permute the vectors via continuous paths while remaining NOD. The next lemma demonstrates that continuous permutations can be performed for a particular NOD frame consisting of the union of two orthonormal bases.

\begin{lemma}\label{twoNpermute}
Suppose $d\geq 3$, $N=2d$, and fix $F\in\mathcal{F}_{d-2,d-3}^\mathbb{R}$ (in the case $d=3$, this is vacuous). Then the following are true:
\begin{itemize}
\item[(i)] There is a $\xi\in\{-1,1\}^{d-2}$ such that
\[
U=\left(\begin{array}{rrr}
        \frac{\sqrt{2}}{2} & \frac{\sqrt{2}}{2} & {\bf 0}_{1,d-2}\\
        \frac{\sqrt{2}}{2} & -\frac{\sqrt{2}}{2}& {\bf 0}_{1,d-2} \\
        0 & 0 &\frac{1}{\sqrt{d-2}} \xi\\
        {\bf 0}_{d-3,1} & {\bf 0}_{d-3,1} & \sqrt{\frac{d-3}{d-2}}F
        \end{array}\right),\: V=\left(\begin{array}{rrr}
        \frac{\sqrt{2}}{2} & \frac{\sqrt{2}}{2} & {\bf 0}_{1,d-2}\\
        0 & 0 & -\frac{1}{\sqrt{d-2}} \xi \\
        \frac{\sqrt{2}}{2} & -\frac{\sqrt{2}}{2} & {\bf 0}_{1,d-2}\\
        {\bf 0}_{d-3,1} & {\bf 0}_{d-3,1} & \sqrt{\frac{d-3}{d-2}}F
        \end{array}\right)
\]
are both positively oriented orthonormal bases and
\[
F_\ast = (u_1\: v_1\: u_2\: v_2\: v_3 \cdots v_d\: u_3 \cdots u_d)
\]
is not OD, where $u_i$ and $v_i$ are the $i$th columns of $U$ and $V$ respectively.

\item[(ii)] There is a continuous path through the NOD members of $\mathcal{F}_{2d,d}^\mathbb{R}$ which connects $F_\ast$ to
\[
G_\ast = (v_1\: u_1\: u_2 \: v_2\: v_3 \cdots v_d\: u_3 \cdots u_d).
\]
\end{itemize}
\end{lemma}

\begin{proof}
First, we let $\xi$ be a member of $\mathcal{F}_{d-2,1}^\mathbb{R}$ which correponds to the Naimark complement of $F$. Now, $\xi$ is unique up to a global sign factor, so we choose the sign so that 
\[
\left(\begin{array}{r}
        \frac{1}{\sqrt{d-2}}\xi\\
        \sqrt{\frac{d-3}{d-2}}F
        \end{array}\right)
\]
is a negatively oriented orthonormal basis. Thus the resulting $U$ and $V$ are positively oriented orthonormal bases. Now, note that $u_1$ has a nonzero inner product with each member of $v_i$ and hence $F_\ast$ is not OD. 

We now describe the continous path connecting $F_\ast$ to $G_\ast$. First, we note that the frame operator of the collection $F_\ast^1=\{u_1,v_1,u_2,v_2\}$ is $\text{diag}(2,1,1,0,\ldots,0)$ and thus the frame operator of $F_\ast^2=\{u_i\}_{i=3}^d\cup\{v_i\}_{i=3}^d$ is $\text{diag}(0,1,1,2,\ldots,2)$. Consequently, we may independently rotate the members of both $F_\ast^1$ and $F_\ast^2$ in the plane spanned by the standard orthonormal vectors $e_2$ and $e_3$, and the resulting frame is still always in $\mathcal{F}_{2d,d}^\mathbb{R}$ and is NOD. Our first action is to continuously rotate $F_\ast^2$ in the span of $e_2$ and $e_3$ to arrive at the frame
\[
\left(\begin{array}{rrrrrr}
        \frac{\sqrt{2}}{2} & \frac{\sqrt{2}}{2} & \frac{\sqrt{2}}{2} & \frac{\sqrt{2}}{2} & {\bf 0}_{1,d-2} & {\bf 0}_{1,d-2}\\
        \frac{\sqrt{2}}{2} &  0 & -\frac{\sqrt{2}}{2}& 0 &\frac{1}{\sqrt{2(d-2)}} \xi & -\frac{1}{\sqrt{2(d-2)}} \xi \\
        0 & \frac{\sqrt{2}}{2} & 0 & -\frac{\sqrt{2}}{2} &\frac{1}{\sqrt{2(d-2)}} \xi & \frac{1}{\sqrt{2(d-2)}} \xi\\
        {\bf 0}_{d-3,1} & {\bf 0}_{d-3,1} & {\bf 0}_{d-3,1} & {\bf 0}_{d-3,1} & \sqrt{\frac{d-3}{d-2}}F & \sqrt{\frac{d-3}{d-2}}F
        \end{array}\right).
\]
The reason for doing this is to avoid any OD frames during a motion that will only involve the first four vectors. At the end of the motion involving only the first four vectors, we shall undo this rotation.

We now restrict our attention to the frame
\[
\left(\begin{array}{rrrr}
        \frac{\sqrt{2}}{2} & \frac{\sqrt{2}}{2} & \frac{\sqrt{2}}{2} & \frac{\sqrt{2}}{2} \\
        \frac{\sqrt{2}}{2} &  0 & -\frac{\sqrt{2}}{2}& 0 \\
        0 & \frac{\sqrt{2}}{2} & 0 & -\frac{\sqrt{2}}{2}
        \end{array}\right).
\]
\noindent By a continuous rotation of $\pi/2$ radians in the plane spanned by $e_2$ and $e_3$, we may move to the frame
\[
\left(\begin{array}{rrrr}
        \frac{\sqrt{2}}{2} & \frac{\sqrt{2}}{2} & \frac{\sqrt{2}}{2} & \frac{\sqrt{2}}{2} \\
        0 & -\frac{\sqrt{2}}{2} &  0 & \frac{\sqrt{2}}{2} \\
        \frac{\sqrt{2}}{2}& 0 & -\frac{\sqrt{2}}{2} & 0 
        \end{array}\right).
\]
This has permuted the columns by a 4-cycle. We now produce a $3$-cycle on the last three columns. First, we may continuously rotate the orthonormal pair consisting of the second and fourth columns by Lemma \ref{spinning} until we arrive at the frame
\[
\left(\begin{array}{rrrr}
        \frac{\sqrt{2}}{2} & 1 & \frac{\sqrt{2}}{2} & 0 \\
        0 & 0 &  0 & 1 \\
        \frac{\sqrt{2}}{2}& 0 & -\frac{\sqrt{2}}{2} & 0 
        \end{array}\right).
\]
Now, the third and fourth columns are orthonormal and we use Lemma \ref{spinning} again to continuously rotate to
\[
\left(\begin{array}{rrrr}
        \frac{\sqrt{2}}{2} & 1 & 0 &\frac{\sqrt{2}}{2} \\
        0 & 0 & -1 &0  \\
        \frac{\sqrt{2}}{2}& 0 & 0 &-\frac{\sqrt{2}}{2}  
        \end{array}\right).
\]
One final application of Lemma \ref{spinning} on the second and third columns yields
\[
\left(\begin{array}{rrrr}
        \frac{\sqrt{2}}{2} & \frac{\sqrt{2}}{2} & \frac{\sqrt{2}}{2} &\frac{\sqrt{2}}{2} \\
        0 & \frac{\sqrt{2}}{2} & -\frac{\sqrt{2}}{2} &0  \\
        \frac{\sqrt{2}}{2}& 0 & 0 &-\frac{\sqrt{2}}{2}  
        \end{array}\right).
\]
Clearly, this path lifts to a continuous path taking the collection $\{u_1,v_1,u_2,v_2\}$ to the collection $\{v_1,u_1,u_2,v_2\}$ such that all intervening frames have unit-norm members and frame operator $\text{diag}(2,1,1,0,\ldots,0)$.

In this paragraph, we justify why this path never meets an OD frame. First, we note that the projection of $F_\ast^1$ into the span of $\{e_2,e_3\}$ contains an orthogonal pair and the inner product of any vector in the rotation of $F_\ast^2$ with any vector in $F_\ast^1$ is equal to the inner products of the vectors if they are first projected into the span of $\{e_2,e_3\}$. Now, simply rotating in the span of $\{e_2,e_3\}$ by some $Q$, it is easy to see that $QF_\ast^1$ still contains an orthogonal pair and thus the rotated vectors from $F_\ast^2$ shall have nonzero inner product with at least one vector from this pair. Using this fact and the fact that $QF_\ast^1$ always has a connected correlation network, we shall have that the entire frame has a connected correlation network and is hence NOD. Furthermore, note that $v_1$ persists as a vector throughout the remaining operations. Since $v_1$ has nonzero correlation with all the rotations of the vectors in $F_\ast^2$ and $e_1$, $v_1$ together with the rotations of the vectors in $F_\ast^2$ has a connected correlation network and forms a frame for $\mathbb{R}^d$. Consequently, the full frame is never OD throughout the continuous motions described above. 

Finally, we undo the starting rotation on $F_\ast^2$ to arrive at $G_\ast$.\end{proof}

\begin{theorem}
The set of NOD frames in $\mathcal{F}_{2d,d}^\mathbb{R}$ is path-connected for $d\geq 3$.
\end{theorem}

\begin{proof}
The proof occurs in three parts. First, we show that any NOD frame connects to a frame in $\mathcal{F}_{2d,d}^\mathbb{R}$ containing a member of $\mathcal{F}_{d,d-1}^\mathbb{R}$ in a $(d-1)$-dimensional subspace of $\mathbb{R}^d$, which is necesarily NOD. In the next step, we show that there is a relatively simple motion that takes any frame in $\mathcal{F}_{2d,d}^\mathbb{R}$ containing a member of $\mathcal{F}_{d,d-1}^\mathbb{R}$ in a $(d-1)$-dimensional subspace to a NOD union of two orthonormal bases. Finally, we show that the NOD unions of two orthonormal bases form a path-connected set.

Suppose $F\in\mathcal{F}_{2d,d}^\mathbb{R}$ is NOD. Let $\mu$ denote the eigensteps of a frame in $\mathcal{F}_{2d,d}^\mathbb{R}$ such that the first $d$ vectors of the frame form a member of $\mathcal{F}_{d,d-1}^\mathbb{R}$ in a $(d-1)$-dimensional subspace of $\mathbb{R}^d$. Using the same permutation argument as in Theorem \ref{realconnect}, we may essentially employ Lemma \ref{lift} to connect $F$ to a frame containing a member of $\mathcal{F}_{d,d-1}^\mathbb{R}$ in $(d-1)$-dimensional subspace of $\mathbb{R}^d$ without crossing an OD frame. 

Now, let $\{u_1,\ldots,u_d\}$ denote the frame vectors forming the member of $\mathcal{F}_{d,d-1}^\mathbb{R}$ in a $(d-1)$-dimensional subspace and let $\{v_1,\ldots,v_d\}$ denote the remaining vectors in the frame. By continuous rotation, we may assume (without loss of generality) that $\text{span}\{u_1,\ldots u_d\}=\text{span}\{e_2,\ldots e_d\}$. Thus, the frame operator of $\{v_1,\ldots,v_d\}$ is 
\[
\text{diag}\left(2,\frac{d-2}{d-1},\ldots,\frac{d-2}{d-1}\right),
\]
and thus there is a $\xi\in\{1,-1\}^d$ and an $H\in\mathcal{F}_{d,d-1}^\mathbb{R}$ such that the coordinate representation of $\{v_1,\ldots,v_d\}$ is given by
\[
\left(\begin{array}{c}
        \sqrt{\frac{2}{d}}\xi\\
        \sqrt{\frac{d-2}{d}}H
        \end{array}\right).  
\]
Let us identify $\{u_1,\ldots,u_d\}$ with $H^\prime\in\mathcal{F}_{d,d-1}^\mathbb{R}$, and let $\zeta\in\{-1,1\}^d$ be a Naimark complement of $H^\prime$. Thus, the coordinate representation of $\{u_1,\ldots,u_d\}$ is given by
\[
\left(\begin{array}{c}
        0\cdot \zeta\\
        H^\prime
        \end{array}\right).  
\]
when viewed as a member of $\mathcal{F}_{d+1,d}^\mathbb{R}$ in $\text{span}\{e_2,\ldots,e_d\}$. By continuous rotation, we may assume that 
\[
\frac{\zeta_1\xi_i}{d}+\frac{d-1}{d}\langle h_1^\prime, h_i\rangle \not = 0
\]
where $h_1^\prime$ is the first vector of $H^\prime$ and $h_i$ is $i$th vector of $H$ for all $i=1,\ldots, d$. This is because $h_1^\prime$ may be rotated to any point on the unit sphere in $(d-1)$ dimensions, and the intersection of the sphere with the set-theoretic complement of any finite number of hyperplanes is always dense in the sphere. We now use the following path:
\[
V(t)=\left(\begin{array}{c}
        \sqrt{\frac{2-t}{d}}\xi\\
        \sqrt{\frac{d-2+t}{d}}H
        \end{array}\right) \text{ and } U(t) = \left(\begin{array}{c}
        \sqrt{\frac{t}{d}}\xi\\
        \sqrt{\frac{d-t}{d}}H^\prime
        \end{array}\right).
\]
By construction, $V(0)=\{v_1,\ldots, v_d\}$ and $U(0)=\{u_1,\ldots,u_d\}$, and $V(1)$ and $U(1)$ are both orthonormal bases. Moreover, the frame operators of $V(t)$ and $U(t)$ are
\[
\text{diag}\left(2-t,\frac{d-2+t}{d-1},\ldots,\frac{d-2+t}{d-1}\right) \text{ and } \text{diag}\left(t,\frac{d-t}{d-1},\ldots,\frac{d-t}{d-1}\right)
\]
so the union of $V(t)$ and $U(t)$ always forms a member of $\mathcal{F}_{2d,d}^\mathbb{R}$. By construction, we also have that $V(t)$ is an equiangular set with nonzero mutual inner products until $t=1$. Thus, $V(t)$ is NOD for $t\in[0,1)$ and hence the union of $V(t)$ and $U(t)$ is NOD for $t\in[0,1)$. By our choice of $h_1^\prime$, $h_1^\prime(1)$ has nonzero inner product with all of the vectors in $V(1)$, and hence we also get that the union of $V(1)$ and $U(1)$ is also NOD. This completes the second step of our proof.

Our final goal is to demonstrate that the set of all NOD unions of two orthonormal bases is connected. This task is divided into two parts. First, we show that if $\{i_k\}_{k=1}^d$ and $\{j_k\}_{k=1}^d$ form a partition of $[2d]$ and $(a,b)\in\{-1,1\}^2$ and if
\[
\mathcal{G}(\{i_k\}_{k=1}^d,\{j_k\}_{k=1}^d,a,b)
\]
is the set of all NOD frames $G=\{g_i\}_{i=1}^{2d}$ with $\{g_{i_k}\}_{k=1}^d$ and $\{g_{j_k}\}_{k=1}^d$ both orthonormal bases having orientations $a$ and $b$ respectively, then $\mathcal{G}(\{i_k\}_{k=1}^d,\{j_k\}_{k=1}^d,a,b)$ is path-connected. The final step is to then use Lemma \ref{twoNpermute} to permute the frame vectors, thus arriving at a NOD frame for which the first $d$ vectors are the standard orthonormal basis and the last $d$ vectors form another orthonormal basis. 

Let $G\in\mathcal{G}(\{i_k\}_{k=1}^d,\{j_k\}_{k=1}^d,a,b)$. For convenience, we first permute so that $\{i_k\}_{k=1}^d=\{k\}_{k=1}^d$ and $\{j_k\}_{k=1}^d=\{k\}_{k=d+1}^{2d}$ and $\{g_i\}_{i=1}^d$ and $\{g_i\}_{i=d+1}^{2d}$ are both positively oriented orthonormal bases. If we construct a path for this configuration, then we may invert the index permutation over the entire path to get a path through $\mathcal{G}(\{i_k\}_{k=1}^d,\{j_k\}_{k=1}^d,a,b)$. Additionally, by continuous rotation, we may assume that $\{g_i\}_{i=1}^d$ is the standard orthonormal basis. Now, consider $g_{d+1}$. There is a continuous rotation $U(t)$ and an $\varepsilon>0$ such that $(U(t)g_{d+1})_j\not = 0$ for all $j=1,\ldots,d$, and $g_{ji}\not = 0$ implies $(U(t)g_i)_j\not = 0$ for all $i=d+1,\ldots, 2d$ and all $t\in(0,\varepsilon)$. Thus, without loss of generality, we may assume that $g_{j, d+1}$ are all nonzero. We now show that we may assume that all of $g_{j,d+1}$ are strictly positive. Suppose $g_{j, d+1}$ is negative, and choose another coordinate index $i\not = j$. Without moving $g_{d+1}$, rotate $g_{d+2}$ so that its nonzero projection onto the span of $e_i$ and $e_j$ is not orthogonal or parallel to the projection of $g_{d+1}$ onto the span of $e_i$ and $e_j$. This is possible because $d>3$ and all of the entries of $g_{d+1}$ are nonzero, and since $g_{d+1}$ stays fixed and has all nonzero entries, the full frame remains NOD while $g_{d+2}$ rotates. At this point, we continuously rotate $\{g_i\}_{i=d+1}^{2d}$ in the span of $e_1$ and $e_2$ until the $i$ and $j$ entries of $g_{d+1}$ are strictly positive. Since the $j$th entry was negative, the intermediate value theorem tells us that this entry becomes zero at some point during this rotation. Thus, $e_i$ or $e_j$ may become orthogonal to $g_{d+1}$ at some point. However, our positioning of $g_{d+2}$ ensures that $g_{d+2}$ will have nonzero inner product with both $e_i$ and $e_j$ at these points. Thus, the entire frame remains NOD during this procedure. Once all of the entries of $g_{d+1}$ are all strictly positive, we continuously rotate $g_{d+1}$ to the vector $\frac{1}{\sqrt{d}}{\bf 1}_d$ while keeping all of the entries of $g_{d+1}$ strictly positive and hence the full frame remains NOD during this procedure. We are now done if we have path-connectivity of the set of frames such that $g_i=e_i$ for $i=1,\ldots, d$, $g_{d+1}=\frac{1}{\sqrt{d}}{\bf 1}_d$, and such that $\{g_i\}_{i=d+1}^{2d}$ is a positively oriented orthonormal basis. All of these frames are NOD, and the path-connectivity follows from the connectivity of $\mathcal{SO}(d-1)$. 

Since the union of the two orthonormal bases in Lemma \ref{twoNpermute} is NOD, we may swap two vectors between the orthonormal pairs by properly permuting the order of the collections, connecting to the frame exhibited in Lemma \ref{twoNpermute} using our above connectivity result, and then peforming the continuous swapping in Lemma \ref{twoNpermute}. Undoing the carefully chosen permutation produces the desired swapping of frame vectors without passing through OD frames. We then swap until the first $d$ and last $d$ vectors form two positively oriented orthonormal bases, and then connect this to a frame with the standard orthonormal basis as the first $d$ vectors and the constant vectors as the $(d+1)$th vector without passing through the OD frames. Since this set of frames is path-connected and contains no OD frames, and we can continuously connect any NOD unions of orthonormal bases by a path through NOD unions of orthonormal bases. This completes the proof.\end{proof}

\begin{example}
Here we give an example for the motion  in $\mathcal{F}_{6,3}^\mathbb{R}$. In Figure \ref{fig2} we see the motion from a frame consisting of a member of $\mathcal{F}_{3,2}^\mathbb{R}$ in the $x$-$y$ plane and the subframe with frame operator $\text{diag}(1,1,2)$. Without passing through the OD frames, we pull the vectors of the member of $\mathcal{F}_{3,2}^\mathbb{R}$ up and the subframe vectors down to get the second frame, which is a union of two orthonormal bases.

\begin{center}
\begin{figure}[h]
\subfigure[]{
\includegraphics[scale=0.3]{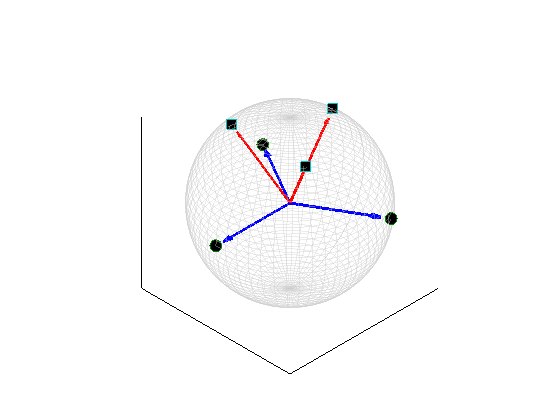}
\label{onb01}
}
\subfigure[]{
\includegraphics[scale=0.3]{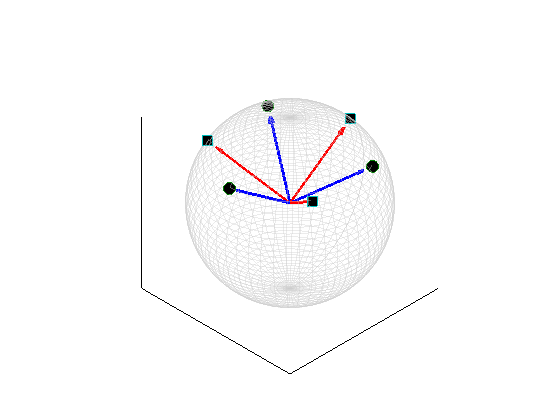}
\label{onb02}
}
\caption{Subfigure (a) illustrates the starting point: a union of a member of $\mathcal{F}_{3,2}^\mathbb{R}$ for a two-dimensional subspace of $\mathbb{R}^3$ and a subframe with frame operator $\text{diag}(1,1,2)$. Subfigure (b) shows the frame obtained by moving the vectors of the member of $\mathcal{F}_{3,2}^\mathbb{R}$ towards the top pole of the sphere and pushing the other vectors away.}
\label{fig2}
\end{figure}
\end{center}

\noindent For the swapping phase of the motion, we first align the orthonormal basis which complements the subframe with frame operator $\text{diag}(1,1,2)$. If we did not do this, the ensuing motion would pass through an OD frame. This is illustrated in Figure \ref{fig3}. Finally, Figure \ref{fig4} shows how the swapping motion uses successive spins of subsets to swap the position of the vectors labeled $\bigcirc$ and $\triangle$.

\begin{center}
\begin{figure}[h]
\subfigure[]{
\includegraphics[scale=0.3]{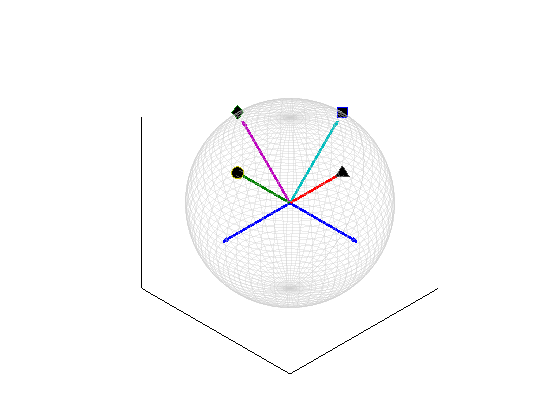}
\label{fm01}
}
\subfigure[]{
\includegraphics[scale=0.3]{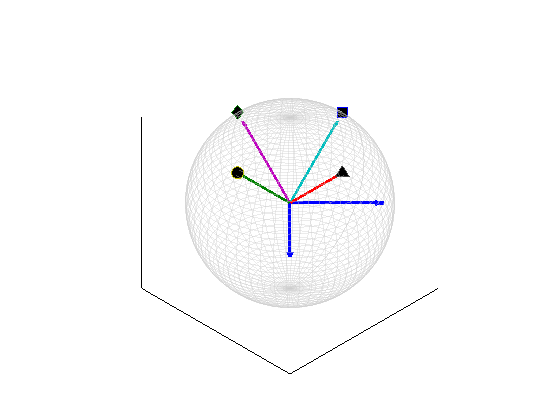}
\label{fm02}
}
\caption{Subfigure (a) indicates the starting position. Subfigure (b) shows how the orthonormal pair spins so that the remaining motions never pass through an OD frame.}
\label{fig3}
\end{figure}
\end{center}
\begin{center}
\begin{figure}[h]
\subfigure[]{
\includegraphics[scale=0.38]{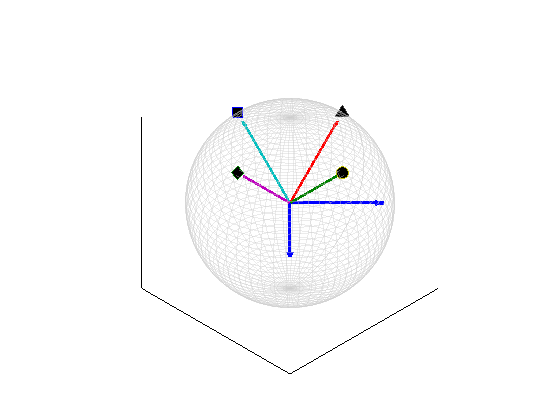}
\label{fm03}
}
\subfigure[]{
\includegraphics[scale=0.38]{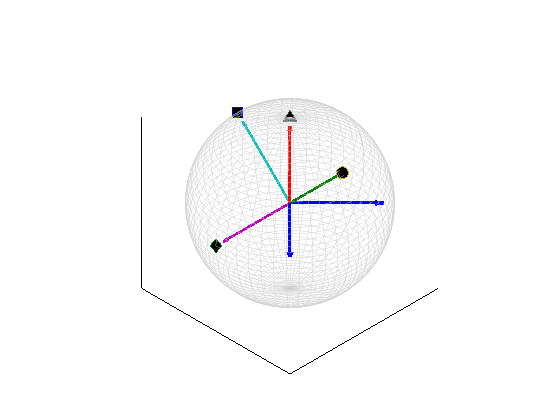}
\label{fm04}
}
\subfigure[]{
\includegraphics[scale=0.38]{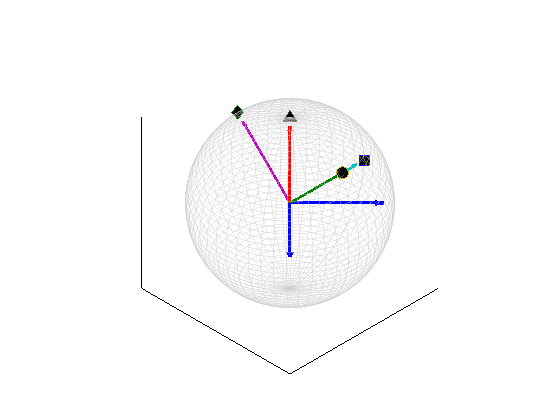}
\label{fm05}
}
\subfigure[]{
\includegraphics[scale=0.38]{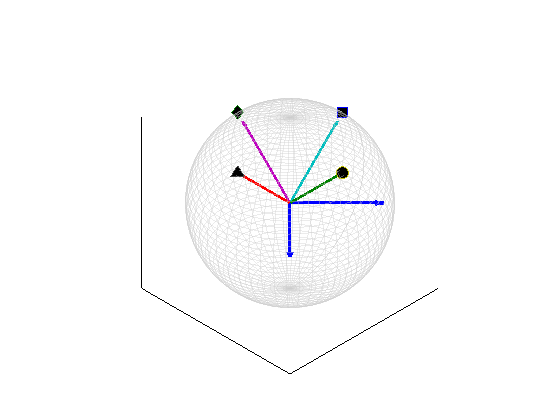}
\label{fm06}
}
\caption{Subfigure (a) shows the spinning of the vectors labeled $(\triangle\:\Box\:\Diamond\:\bigcirc)$ in Figure \ref{fig3}(b) to the ordering $(\bigcirc\:\triangle\:\Box\:\Diamond)$. The remaining figures illustrate how to perform a cycle of $(\triangle\:\Box\:\Diamond)$ to end at the ordering $(\bigcirc\:\Box\:\Diamond\:\triangle)$. In Subfigures (b) through (d), we spin $(\triangle\:\Box)$, $(\Diamond\:\Box)$, and finally $(\Box\:\triangle)$.}
\label{fig4}
\end{figure}
\end{center}

\end{example}

\noindent These last corollaries summarize all of the results of this section.

\begin{corollary}\label{cmplxNODconnect}
If $N$ and $d$ satisfy $N\geq d\geq 1$, then the NOD frames in $\mathcal{F}_{N,d}^\mathbb{C}$ form a path-connected set.
\end{corollary}

\begin{corollary}\label{realNODconnect}
 If $N$ and $d$ satisfy $N\geq d+2>4$ or $d=2$ and $N\geq 5$, then the NOD frames in $\mathcal{F}_{N,d}^\mathbb{R}$ form a path-connected set.
\end{corollary}

\section{Irreducibility of $\mathcal{F}_{N,d}^{\mathbb{F}}$ and consequences} 

Combining Corollaries~\ref{cmplxNODconnect} and~\ref{realNODconnect} and Lemma \ref{nod-irreducible}, Theorem \ref{irreducible} follows. In this section, we explore an interesting consequence of this irreducibility. In particular, we turn our attention to the demonstration of Theorem~\ref{fullspark}. By Proposition~\ref{realgeneric}, finishing our demonstration that a generic frame in $\mathcal{F}_{N,d}^\mathbb{F}$ is full spark just requires that there exists a full spark frame in $\mathcal{F}_{N,d}^\mathbb{F}$.

\begin{theorem}[Theorems~4 and~5 in~\cite{PK05}]\label{fullsparkexists}
For every $d$ and every $N\geq d$, there exists a full spark FUNTF $\{f_n\}_{n=1}^N\subseteq\mathbb{R}^d$.
\end{theorem}

\begin{proof}[Proof of Theorem \ref{fullspark}]
By Theorem \ref{fullsparkexists}, the set of full spark frames in $\mathcal{F}_{N,d}^\mathbb{F}$ is nonempty, and hence the real generic property holds by Proposition \ref{realgeneric} since the nonsingular points of $\mathcal{F}_{N,d}^\mathbb{F}$ form a connected dense subset (Corollaries~\ref{cmplxNODconnect} and~\ref{realNODconnect}, the discussion preceding Lemma \ref{nod-irreducible}, and Lemma~\ref{nod-dense}).\end{proof}

\section{Discussion}
Proposition~\ref{realgeneric} is often identified with the additional property that $V\cap U$ is either a null set or has full measure in $V$. This additional property presupposes the existence of a uniform measure on $\mathcal{F}_{N,d}^\mathbb{F}$. If the algebraic variety happens to also be a manifold, we can be sure that this uniform distribution exists (see \cite{Stroock}, for example). On the other hand, in private communications with Christopher Manon of George Mason University, it has been suggested that the theory of Duistermaat-Heckman measures \cite{DH82} allows us to induce a uniform distribution on $\mathcal{F}_{N,d}^\mathbb{F}$ by (1) drawing a point uniformly from the the polytope of eigensteps, (2) uniformly drawing from the possible $U_1$ and the $V_n$ in Theorem 7 of \cite{CFMPS12}, and (3) reconstructing the frame from the iteration from the data provided in (1) and (2), as indicated by Theorem 7 of \cite{CFMPS12}. Because this procedure is still under active research, we only definitively say that the full spark frames have full measure in the uniform measure of $\mathcal{F}_{N,d}^\mathbb{F}$ when $N$ and $d>2$ are relatively prime, and we leave the full examination of this question as a future project.

Another interesting question is whether these results can be extended to the infinite-dimensional setting. Since the eigensteps construction is our primary tool, the first step of this process would involve a generalization of this construction. One could also study whether similar results hold for sets of frames with different frame operators and different norms of the frame vectors.

\section*{Acknowledgements}
We would like to thank Bernhard Bodmann, Gitta Kutyniok, and Tim Roemer for organizing the American Institute of Mathematics workshop ``Frame Theory intersects Geometry" where this work began. We thank the American Institute of Mathematics for their great generosity. We also would like to thank Eva-Maria Feichtner and Emily King for organizing the workshop ``Frames and Algebraic \& Combinatorial Geometry," which has brought more light to the application of algebraic geometry in frame theory.  J.\ Cahill was supported by NSF Grant No.\ ATD-1321779. D.\ G.\ Mixon was supported by an AFOSR Young Investigator Research
Program award, NSF Grant No.\ DMS-1321779, and AFOSR Grant No.\ F4FGA05076J002. N.\ Strawn was supported by NSF Grant No.\ DMS-10-45153. The views expressed in this article are those of the authors and do not reflect the official policy or position of the United States Air Force, Department of Defense, or the U.S. Government.

\bibliographystyle{amsplain}

\end{document}